\documentclass[twoside,a4paper,11pt]{article} 
\usepackage[T1]{fontenc}
\usepackage{palatino}
\usepackage[utf8]{inputenc} 

\usepackage{booktabs} 
\usepackage{array} 
\usepackage{paralist} 
\usepackage{verbatim} 
\usepackage{subfig} 
\usepackage[backend=bibtex,style=alphabetic,sorting=nyvt]{biblatex}
\usepackage{amssymb}
\usepackage[dvipsnames]{xcolor}
\usepackage{mathtools}
\mathtoolsset{showonlyrefs}
\usepackage[scr=rsfs,cal=boondox]{mathalfa}
\usepackage{amsthm}
\usepackage{tikz-cd}
\usepackage[colorlinks, allcolors=blue]{hyperref}
\usepackage{titlesec}
\usepackage{graphicx} 
\usepackage{calrsfs}
\DeclareMathAlphabet{\pazocal}{OMS}{fszplm}{m}{n}
 \usetikzlibrary{calc}

\usepackage{geometry} 
\usepackage{fancyhdr} 

\titleformat{\section}
       { \centering \bfseries \large}
       {\thesection.}
       {0.5em}
       {}
       []

\titleformat{\subsection}[runin]
       {\normalfont\bfseries}
       {\thesubsection.}
       {0.5em}
       {}
       [.]
\titleformat{\subsubsection}[runin]
       {\normalfont\bfseries}
       {\thesubsubsection.}
       {0.5em}
       {}
       [.]

\usepackage[titles,subfigure]{tocloft} 

\newcommand{\Ghat}{\widehat{G}}

\newcommand{\YBS}{\overline{Y}^\textrm{BS}}
\newcommand{\YBSN}{\overline{Y_1(N)}^\textrm{BS}}

\newcommand{\new}{\textrm{new}}
\newcommand{\old}{\textrm{old}}

\newcommand{\8}{\infty}

\newcommand{\Z}{\mathbb{Z}}
\newcommand{\NN}{\mathbb{N}}

\newcommand{\id}{\mathbf{1}}

\newcommand{\xbf}{\mathbf{x}}

\newcommand{\vbf}{\mathbf{v}}
\newcommand{\wbf}{\mathbf{w}}

\newcommand{\Span}{\mathrm{span}}
\newcommand{\SO}{\mathrm{SO}}

\newcommand{\SL}{\mathrm{SL}}

\newcommand{\KM}{\mathrm{KM}}
\newcommand{\BM}{\mathrm{BM}}

\newcommand{\GL}{\mathrm{GL}}

\newcommand{\D}{\mathbb{D}}

\newcommand{\HH}{\mathbb{H}}

\newcommand{\R}{\mathbb{R}}

\newcommand{\Q}{\mathbb{Q}}
\newcommand{\C}{\mathbb{C}}

\newcommand{\PP}{\mathbb{P}}

\newcommand{\q}{\mathfrak{q}}

\newcommand{\ffrak}{\mathfrak{f}}

\newcommand{\Ehat}{\widehat{E}}

\newcommand{\vphat}{\widehat{\varphi}}

\newcommand{\PD}{\mathrm{PD}}

\newcommand{\tr}{\mathrm{tr}}

\newcommand{\Ocal}{\mathcal{O}}
\newcommand{\Scal}{\pazocal{S}}

\newcommand{\Fcal}{\pazocal{F}}
\newcommand{\Hcal}{\pazocal{H}}

\newcommand{\Kcal}{\pazocal{K}}

\newcommand{\Tcal}{\pazocal{T}}
\newcommand{\Ccal}{\pazocal{C}}

\newcommand{\Bcal}{\mathcal{B}}

\newcommand{\Pcal}{\pazocal{P}}

\newcommand{\Zcal}{\pazocal{Z}}

\newcommand{\Ecal}{\mathcal{E}}

\newcommand{\re}{\mathrm{Re}}

\newcommand{\Mat}{\mathrm{Mat}}
\newcommand{\im}{\mathrm{Im}}
\newcommand{\rk}{\mathrm{rk}}

\newcommand{\hooklongrightarrow}{\lhook\joinrel\longrightarrow}

\newcommand\restr[2]{{
  \left.\kern-\nulldelimiterspace 
  #1 
  \vphantom{\big|} 
  \right|_{#2} 
  }}
 
\newcommand{\Th}{\mathrm{Th}}
\newcommand{\rd}{\textrm{rd}}

\newcommand{\Mcal}{\pazocal{M}}

\newtheoremstyle{mytheoremstyle} 
    {1.5em}                    
    {1.em}                    
    {\itshape}                   
    {}                           
    {\normalsize \bfseries}                   
    {.}                          
    {0,5em}                       
    {}  
    
\theoremstyle{mytheoremstyle}

\newtheorem{thm}{Theorem}[section]
\newtheorem*{thm*}{Theorem}
\newtheorem{cor}{Corollary}[thm]
\newtheorem*{cor*}{Corollary}
\newtheorem{lem}[thm]{Lemma}
\newtheorem{prop}[thm]{Proposition}
\newtheorem{mydef}[thm]{Definition}
\newtheorem*{que*}{Question}

\theoremstyle{remark}
\newtheorem{rmk}{Remark}[section]

\numberwithin{equation}{section}

\renewenvironment{abstract}
 {\small
  \begin{center}
  \bfseries \abstractname\vspace{-.5em}\vspace{0pt}
  \end{center}
  \list{}{
    \setlength{\leftmargin}{2cm}%
    \setlength{\rightmargin}{\leftmargin}%
  }%
  \item\relax}
 {\endlist}

\begin{document}
\title{Theta correspondence and the Borisov-Gunnells relations}
\author{ Romain Branchereau }
\date{} 
\maketitle

\begin{abstract} We consider a geometric theta correspondence from the first homology of a modular curve, to modular forms of weight $2$. Using Stevens' description of the homology, we find that this map sends modular symbols to product of weight one Eisenstein series, modular caps to weight $2$ Eisenstein series, and hyperbolic cycles to diagonal restrictions of Hilbert-Eisenstein series. We use it to revisit work of Borisov and Gunnells, and explain its connection to a theorem of Li. In particular, we give a geometric proof of certain relations between Eisenstein series.
\end{abstract}

{
  \tableofcontents \vspace{1cm}
}

\section{Introduction}

For a congruence subgroup $\Gamma \subseteq \SL_2(\Z)$, let $M_k(\Gamma)$ be the space of holomorphic modular forms of weight $k$, and $S_k(\Gamma)$ be the subspace of cusp forms. Its orthogonal complement with respect to the Petersson inner product is the space $E_k(\Gamma)$ of holomorphic Eisenstein series.

For an integer $r$, we define the Eisenstein series of weight $k$
\begin{align}
G_r^{(k)}(\tau) \coloneqq N \frac{(k-1)!}{(-2i\pi)^k} \lim_{s \rightarrow 0}   \sideset{}{^\prime}\sum_{m,n \in\Z} \frac{1}{(mN\tau+n)^k \vert mN\tau+n \vert^{2s}}e\left ( -\frac{rn}{N}\right )
\end{align}
by analytic continuation. We will also need the Eisenstein series 
\begin{align}
\Ghat_r^{(k)}(\tau) \coloneqq \frac{(k-1)!}{(-2i\pi)^k} \lim_{s \rightarrow 0}   \sideset{}{^\prime}\sum_{\substack{m,n \in\Z \\ m \equiv r \mod{N}}} \frac{1}{(m\tau+n)^k \vert m\tau+n \vert^{2s}}.
\end{align}
Both define holomorphic Eisenstein series in $E_k(\Gamma_1(N))$, except in some cases when the weight is $k=2$. To obtain a holomorphic Eisenstein series of weight $2$, we set
\begin{align}
H_{p,q}^{(2)} \coloneqq G^{(2)}_{q}-\delta_{q0}\Ghat^{(2)}_{p} \in E_2(\Gamma_1(N))
\end{align}
for two integers $p$ and $q$.

 \subsection{A theta lift} In this paper, we will be interested in the congruence subgroup $\Gamma=\Gamma_1(N)$ for $N \geq 4$, and its associated modular curve $Y_1(N) \coloneqq \Gamma_1(N) \backslash \HH$. In previous work \cite{rbrsln}, we combined ideas due to Bergeron-Charollois-Garcia \cite{bcg,bcgcrm} and Kudla-Millson \cite{KM86,KM87,KM90} to construct a closed differential form
\begin{align} \label{form 1 intro}
\Ecal(z,\tau) \in \Omega^1(Y_1(N))\otimes C^\8(\HH)
\end{align}
that transforms in the variable $\tau$ as a modular form of weight $2$ for $\Gamma_1(N)$. When passing to cohomology, the form becomes holomorphic in $\tau$, and defines a holomorphic modular form valued in the cohomology of the modular curve $Y_1(N)$. More precisely, we have the following modularity result.

\begin{thm} \label{thm 1 intro}Let $T_n\{0,\8\}$ be the Hecke translate of the modular symbol $\{0,\8\}$, and $\PD(T_n\{0,\8\})$ its Poincaré dual in $H^1(Y_1(N);\Z)$. In cohomology, the class of the differential form $\Ecal$ has a Fourier expansion
\begin{align}
 [\Ecal]=-\frac{1}{2i \pi} d \log(g_{0,1})-\sum_{n=1}^\8 \PD(T_n\{0,\8\}) e(n\tau) \in H^1(Y_1(N);\Q) \otimes M_2(\Gamma_1(N)), \qquad e(\tau)=e^{2i\pi \tau},
\end{align}
where $g_{0,1}$ is a Siegel unit.
\end{thm}
The integral of $\Ecal(z,\tau)$ over a $1$-cycle induces a theta lift
\begin{align} \label{liftbcg}
\Ecal \colon H_1(Y_1(N);\C) \longrightarrow M_2(\Gamma_1(N))
\end{align}
to holomorphic modular forms. By passing from homology to cohomology, the map \eqref{liftbcg} can be restricted to the map 
\begin{align}
S_2(\Gamma_1(N)) \longrightarrow S_2(\Gamma_1(N)), \qquad f \longmapsto \sum_{n=1}^\8 \left ( \int_0^{\8}(T_nf)(z)dz \right )e(n\tau),
\end{align}which was considered by Borisov-Gunnells \cite[Theorem.~4.8]{BG01}. 

Let $S^{\new}_{k,\rk=0}(\Gamma) \subseteq S_k(\Gamma)$ be the subspace spanned by newforms of rank $0$, {\em i.e.} for which $L(f,1) \neq 0$. Let $\Hcal^{(2)} \subseteq E_2(\Gamma_1(N))$ be the subspace spanned by the weight $2$ Eisenstein series $H^{(2)}_{p,q}$ with $(p,q) \not \equiv (0,0)$ modulo $N$.
\begin{thm} \label{corintrospan} Let $\im(\Ecal)$ be the image of the theta lift \eqref{liftbcg}. We have the inclusion
\begin{align}
\Hcal^{(2)} \oplus S^{\new}_{2,\rk=0}(\Gamma_1(N)) \subseteq \im(\Ecal),
\end{align}
which is an equality when $N$ is prime.
\end{thm}
Two complementary results in the literature describe spanning sets of $S^{\new}_{2,\rk=0}(\Gamma)$. For $\Gamma=\Gamma_0(N)$, Li proves in \cite{LI17} that this space is spanned by diagonal restrictions of certain Hilbert-Eisenstein series. On the other hand, Borisov-Gunnells show in \cite[Theorem.~4.11]{BG01} that it can also be spanned by Eisenstein series of weight $2$, and products of weight $1$ Eisenstein series. The first goal of this paper is to explain how similar results arise naturally from the above theta lift \eqref{liftbcg}, by evaluating it on appropriate generators of the homology. The second goal is to use this theta lift to obtain relations among Eisenstein series. 

\subsection{Evaluation on hyperbolic matrices and Li's result} The homology $H_1(Y_1(N);\C)$ is spanned by cycles
\begin{align}
\Zcal_\gamma \coloneqq \{z_0,\gamma z_0\}, \qquad \gamma \in  \Gamma_1(N),
\end{align}
for some basepoint $z_0$ is $\HH$. As explained in \cite{rbrsln} in the case of the modular curve $Y_0(p)=\Gamma_0(p) \backslash \HH$, we can recover Li's result by evaluating the theta lift on cycles $\Zcal_\gamma$, with $\gamma$ a hyperbolic matrix. This argument extends naturally to $\Gamma_1(N)$. 

 Let $F=\Q(\sqrt{D})$ be a real quadratic field, with ring of integers $\Ocal_F$. Let $\ffrak \subset F$ be a lattice of rank $2$, and $\Ocal \subseteq \Ocal_F$ be the largest order preserving $\ffrak$. The subgroup of units $U^+ \coloneqq \Ocal^{\times,+} \cap (1+N\ffrak) \subset \Ocal^{\times,+}$ preserves the coset $1+N\ffrak$. We define the Hilbert-Eisenstein series
\begin{align}
E_{1,\ffrak}^{(1)}(\tau,\tau') \coloneqq \frac{D^{\frac{1}{2}}}{4\pi^2}\lim_{s \rightarrow 0} \sideset{}{^\prime}\sum_{(m,n) \in ( 1+N\ffrak) \times N\ffrak / U^+} \frac{(yy')^s}{(m\tau+n)(m'\tau'+n')\vert m\tau+n \vert^{2s} \vert m'\tau'+n'\vert^{2s}},
\end{align}
of parallel weight $1$, for the congruence subgroup $\Gamma_1(N\ffrak) \subset \SL_2(\Ocal)$. Its diagonal restriction (obtained by setting $\tau=\tau'$) is a holomorphic modular form $E_{1,\ffrak}^{(1)}(\tau,\tau) \in M_2(\Gamma_1(N))$.

\begin{thm} \label{previous theorem} Let $\gamma = \begin{psmallmatrix} a & b \\ c & d \end{psmallmatrix}$ be a primitive hyperbolic matrix in $\Gamma_1(N)$, let $D \coloneqq \tr(\gamma)^2-4$,  and $\ffrak_\gamma=\Z+\nu \Z$ with $\nu=\frac{1}{2c}(a-d-\sqrt{D})$. Then
\begin{align}
\Ecal(\Zcal_\gamma) = E_{1,\ffrak_\gamma}^{(1)} \left (\tau,\tau\right ),
\end{align} 
and $\im(\Ecal)=\Span_\C \left \{ E_{1,\ffrak_\gamma}^{(1)}(\tau,\tau) \ \vert \ \gamma \in \Gamma_1(N)\ \textrm{primitive hyperbolic}\right \}$.
\end{thm}
By combining Theorem \ref{previous theorem} with Theorem \ref{corintrospan}, we deduce that
\begin{align}
\Hcal^{(2)} \oplus S^{\new}_{2,\rk=0}(\Gamma_1(N)) \subseteq \Span_\C \left \{ E_{1,\ffrak_\gamma}^{(1)}(\tau,\tau) \ \vert \ \gamma \in \Gamma_1(N)\ \textrm{primitive hyperbolic}\right \},
\end{align}
where the inclusion is an equality when $N$ is prime. This result is similar to \cite[Theorem.~1.2]{LI17} in level $\Gamma_0(N)$, where the author considers different Hilbert-Eisenstein series, defined by their Fourier expansions.

\subsection{Modular caps, modular symbols and the result of Borisov-Gunnells} The main result of this paper is to relate the theta lift \eqref{liftbcg} to the work of Borisov-Gunnells.

We use Stevens' \cite{S89} description of the homology of the modular curve. Let $\YBSN$ be the Borel-Serre compactification of $Y_1(N)$. Its boundary is a union of circles $\Ccal_r$ at each cusp $r \in \Gamma_1(N) \backslash \PP^1(\Q)$, that we will call {\em closed modular caps}. From the long exact sequence in homology, one can write 
\begin{align} \label{decomposition intro}
H_1(Y_1(N);\C) \simeq H_1(\YBSN;\C) \simeq \Ccal(\C)  \oplus \Mcal \Scal_0(\C),
\end{align}
where the space $\Ccal(\C)$ is spanned by modular caps and $\Mcal \Scal_0(\C)$ is spanned by modular symbols of degree $0$. 

The cycles $\Zcal_\gamma$ can be explicitly be written as a linear combination of closed modular caps and unimodular symbols, with respect to the splitting \eqref{decomposition intro}. First, if $\gamma=\begin{psmallmatrix} 1 & n \\ 0 & 1\end{psmallmatrix}$, then it is easy to see that $\Zcal_\gamma$ can be ``pushed'' to the cusp $\8$ and it is equal to $n \Ccal_\8$. The same can be done for any parabolic matrix, after translation. 
 
 \begin{figure}[h] 
\centering
\captionsetup{width=.75\textwidth,font={small,it}}
\includegraphics[scale=0.30]{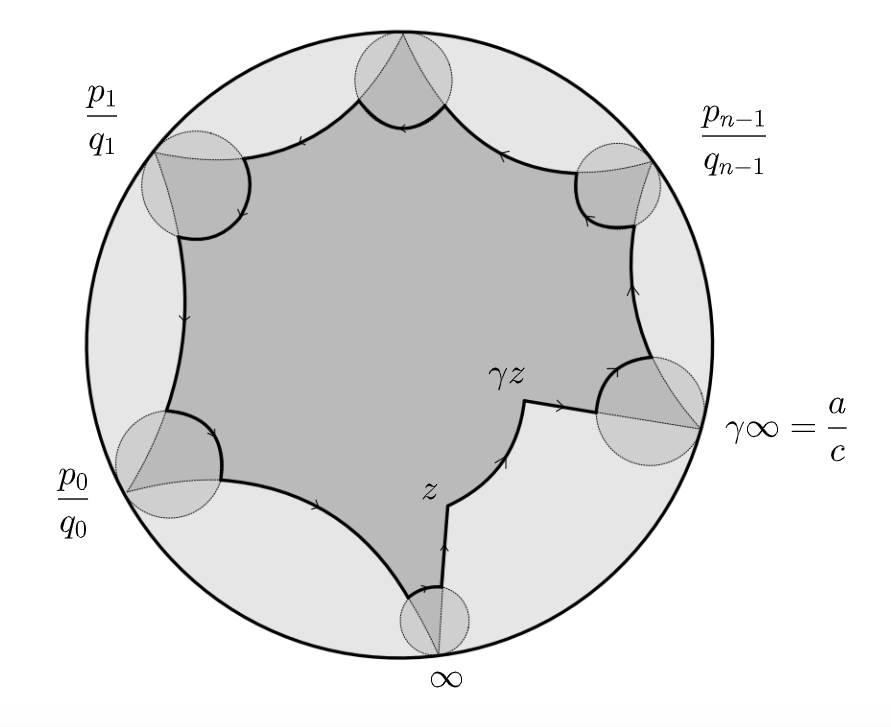}
\caption{ We represent $\HH$ in the Poincaré disc model. The polygon $\Pcal$ is closed in the Borel-Serre compactification. The circles around the cusps are the horocycles at infinity containing the modular caps. The theta lift sends the modular symbols on each side to a product of two weight $1$ Eisenstein series and each modular cap to a weight $2$ Eisenstein series. The two segments between the interior and the cusp cancel out.}
\label{moving to boundary}
\end{figure}

 For hyperbolic matrices, this can be achieved by adapting the continued fraction algorithm explained in \cite{S89}. Suppose that $\gamma=\begin{psmallmatrix} a & b \\ c & d \end{psmallmatrix}$ with $c \neq 0$. Consider the continued fraction expansion of $\frac{a}{c}=\gamma \8$
\begin{align}
\frac{a}{c}= b_0-\cfrac{1}{b_1-\cfrac{1}{\cdots-\cfrac{1}{b_{n-1}-\cfrac{1}{b_n}}}}
\end{align}
with $b_0 \in \Z$, $b_k \in \NN$ for $k \geq 1$, and convergents $\frac{p_0}{q_0}=\frac{b_0}{1}, \frac{p_1}{q_1},\dots ,\frac{p_n}{q_n}=\frac{a}{c}$. Set $(p_{-1},q_{-1})\coloneqq (1,0)$ and
\begin{align}
\gamma_k \coloneqq \begin{pmatrix}
-p_k & p_{k-1} \\ -q_k & q_{k-1}\end{pmatrix}, \qquad k=0, \dots,n.
\end{align} As shown in Figure \ref{moving to boundary}, we can find a polygon $\Pcal$ whose boundary is $\Zcal_\gamma$ plus a linear combination of modular caps and modular symbols. It follows that we can write the cycle in homology as
\begin{align} \label{cycle decomp intro}
\Zcal_\gamma = \left [ (b_0+bq_{n-1}-p_{n-1}d)\Ccal_\8 + \sum_{k=0}^{n-1} b_{k+1} \Ccal_{\gamma_k\8} \right ] +  \sum_{k=0}^n \gamma_k \{ 0,\8\} \in \Ccal(\C) \oplus \Mcal \Scal_0(\C).
\end{align}

Furthermore, we show that the form $\Ecal(z,\tau)$ extends to a closed form
\begin{align}
\Ecal(z,\tau) \in \Omega^1(\YBSN)\otimes C^\8(\HH)
\end{align}
on the Borel-Serre compactification. Hence, we can use the decomposition \eqref{cycle decomp intro} and Stokes theorem on $\YBSN$ to write the integral of $\Ecal(z,\tau)$ over $\Zcal_\gamma$ as a linear combination of integrals over modular caps and modular symbols.

\begin{thm} \label{mainthm} \begin{enumerate} \item Let $\Ccal_r$ be the closed modular cap around the cusp $r=\gamma_r \8$ for some matrix $\gamma_r =\begin{psmallmatrix}
a & b \\ c& d
\end{psmallmatrix} \in \SL_2(\Z)$. The period of $\Ecal(z,\tau)$ over $\Ccal_r$ is the weight $2$ Eisenstein series $H^{(2)}_{d,-c}(\tau)$.
\item Let $\Mcal=\gamma \{0,\8\}$ be the unimodular symbol with $\gamma=\begin{psmallmatrix}
a & b \\ c& d
\end{psmallmatrix}$ in $\SL_2(\Z)$. The period of $\Ecal(z,\tau)$ over $\Mcal$ is the product $-G_d^{(1)}(\tau)G_{c}^{(1)}(\tau)$ of two weight $1$ Eisenstein series.
\end{enumerate}
\end{thm}

Let $\Hcal^{(1,1)} \subseteq M_2(\Gamma_1(N))$ be the span of all products $G^{(1)}_aG^{(1)}_b$, for $a,b \in \Z/N\Z$. From Theorem \ref{mainthm}, the decomposition \eqref{cycle decomp intro} and Stokes theorem, the deduce the following.
\begin{thm} \label{previous theorem 2} If $\gamma = \begin{psmallmatrix} a & b \\ c & d \end{psmallmatrix}$ is a hyperbolic matrix in $\Gamma_1(N)$, then
\begin{align}
\Ecal(\Zcal_\gamma) =(b_0-p_{n-1}d+bq_{n-1})H_{1,0}^{(2)} + \sum_{k=0}^{n-1} b_{k+1}H^{(2)}_{q_{k-1},q_k} - \sum_{k=0}^n G_{q_k}^{(1)}G_{q_{k-1}}^{(1)}.
\end{align}
In particular, we have $\im(\Ecal) \subseteq \Span_\C\{\Hcal^{(2)},\Hcal^{(1,1)}\}$.
\end{thm}
By combining Theorem \ref{previous theorem 2} with Theorem \ref{corintrospan}, we deduce that
\begin{align}
\Hcal^{(2)} \oplus S^{\new}_{2,\rk=0}(\Gamma_1(N)) \subseteq \Span_\C\{\Hcal^{(2)},\Hcal^{(1,1)}\}.
\end{align}
This is similar to \cite[Theorem.~4.11]{BG01}, with slightly different definitions of the Eisenstein series; see Remark \ref{remark comparison}. Note that the result in {\em loc. cit.} is stronger, as it proves an equality.

\begin{figure}[h] 
\centering
\captionsetup{width=.75\textwidth,font={small,it}}
\includegraphics[scale=0.40]{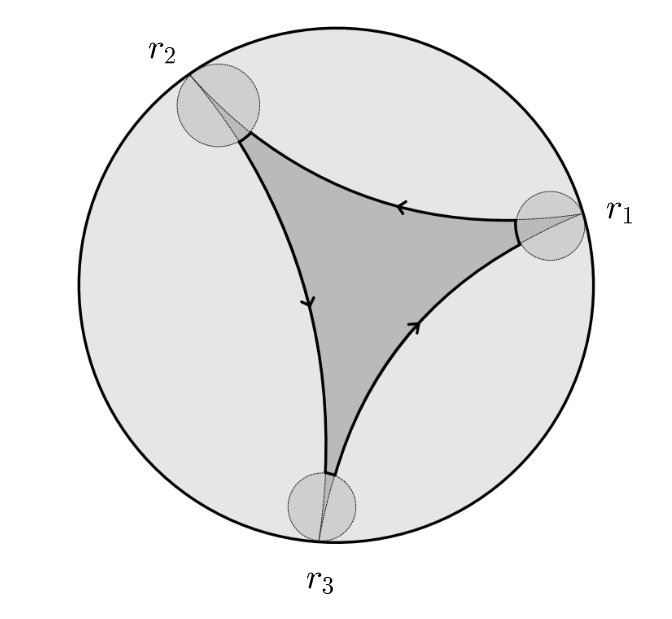}
\caption{The hyperbolic triangle closes in the Borel–Serre compactification. Its vertices $r_1, r_2, r_3$ are cusps connected by unimodular symbols and completed by modular caps. Each side is mapped to a product of two weight $1$ Eisenstein series, while each modular cap is mapped to a weight $2$ Eisenstein series. }
\label{moving to boundary 3}
\end{figure}

\subsection{Relation between Eisenstein series} Finally, we can use the theta lift to deduce relations between Eisenstein series. Let $\Tcal \subset \HH$ be a hyperbolic triangle whose sides are three unimodular symbols and closed by three modular caps, as in Figure \ref{moving to boundary 3}. The image of the boundary $\partial \Tcal$ is trivial in $H_1(Y_1(N);\C)$, so that $\Ecal(\partial \Tcal)=0$. The boundary $\partial \Tcal$ is a linear combination of unimodular symbols and modular caps, whose image have been computed in Theorem \ref{mainthm}. We deduce the following result, similar to \cite[Proposition.~3.7]{BG01}.

\begin{thm} \label{corollary intro}
Let $a,b,c$ be three coprime integers that satisfy $a+b+c \equiv 0 \mod N$ and such that $a,b,c \neq 0 \mod N$. We have
\begin{align}
G_a^{(1)}G_{b}^{(1)}+G_b^{(1)}G_{c}^{(1)}+G_c^{(1)}G_{a}^{(1)}=G_a^{(2)}+G_b^{(2)}+G_c^{(2)}.
\end{align}
\end{thm} 
\begin{rmk}Theorem \ref{corollary intro} generalizes to the case of $d$ Eisenstein series of weight $2$, and $d$ products of two Eisenstein series of weight $1$; see Corollary \ref{corpolygon}.
\end{rmk}

\subsection{Relation to other works}
The theta lift \eqref{liftbcg} is similar to a cocycle constructed in \cite[Theorem.~5]{bcg} or \cite[Théorème p.~2.10]{bcgcrm}, without the geometric interpretation of the Fourier coefficients. The construction in {\em loc. cit.} also uses the Mathai-Quillen formalism \cite{MQ86}, as well as previous work of Charollois-Sczech \cite{CS16}. The cocycle is given by a product of weight one Eisenstein series, so that the case $N=2$ of \cite[Théorème p.~2.10]{bcgcrm} is analogous to the second part of Theorem \ref{mainthm}. A result similar to ours has also been obtained by Xu \cite{X24}, where the cocycle of Bergeron-Charollois-Garcia is computed explictly in terms of Kronecker-Eisenstein series of weights $1$ and $2$, using the Bruhat decomposition. Comparable results in the setting of $\SL_N(\Ocal_K)$ over an imaginary quadratic field can be found in \cite{BCG23} and \cite{rbraif25}. Finally, the connection to the work of Borisov-Gunnells was already suggested in \cite[Example p.~8]{bcg} and \cite[Remark p.~30]{bcgcrm}.

There have been several results on spanning modular forms by Eisenstein series, extending and generalizing the work of Borisov-Gunnells; e.g. by Dickson-Neururer \cite{DN18}, Xue \cite{X23}, or Raum-Xia \cite{RX20}.

Finally, let us mention several further results in the literature concerning relations among Eisenstein series, such as the recent work of Brunault \cite{B25}, Brunault-Zudilin \cite{BZ23}, Khuri-Makdisi--Raji \cite{KMM17}, or Zhang \cite{haozhang}.

\section{Preliminaries on the modular curve}
\subsection{Homology classes} Let $Y \coloneqq \Gamma \backslash \HH$ be the modular curve for some congruence subgroup $\Gamma$ that will be $\Gamma(N)$ or $\Gamma_1(N)$, always with $N \geq 4$. 
By Hurewicz's theorem, the group homorphism $\Gamma \longrightarrow H_1(Y;\Z)$ that sends $\gamma$ to the cycle $\Zcal_\gamma \coloneqq \{z_0, \gamma z_0\}$ is independent of the choice of the basepoint $z_0$ and surjective. Moreover, by the universal coefficient theorem we have
\begin{align}
H_1(Y;\Z) \otimes \C \cong H_1(Y;\C),
\end{align}
so that $H_1(Y;\C)$ is spanned (over $\C$) by the cycles $\Zcal_\gamma$.

The trace of a matrix $\gamma$ in $\Gamma_1(N)$ is congruent to $2$ modulo $N$. In particular, if $N \geq 4$, then $\vert \tr(\gamma) \vert \geq 2$. \begin{mydef}
The matrix is {\em parabolic} if $\vert \tr(\gamma) \vert = 2$, in which case it stabilizes a cusp in $\PP^1(\Q)$.  The matrix is {\em hyperbolic} if $\vert \tr(\gamma)\vert > 2$, and in that case it stabilizes two real quadratic points in $\PP^1(\R) \smallsetminus \PP^1(\Q)$. Moreover, it is {\em primitive} in $\Gamma_1(N)$ if it cannot be written as a nontrivial power $\gamma=\gamma_1^m$ of another hyperbolic matrix $\gamma_1 \in \Gamma_1(N)$.
\end{mydef}

\subsection{Borel-Serre compactification} We follow Stevens \cite{S89}. Let $\PP^1(\Q)$ be the boundary of $\HH$ and let
\begin{align}
\overline{\HH}=\HH \cup \bigsqcup_{r \in \PP^1(\Q)}  B_r 
\end{align}
be the Borel-Serre completion, obtained by gluing a horocycle $B_r=\PP^1(\R) \smallsetminus \{ r\}$ at each cusp $r$. At $\8$, the gluing is done such that a sequence $z_n=x_n+iy_n$ converges to $\alpha \in \PP^1(\R) \smallsetminus \{ \8\} \simeq \R$ if $\lim_{n\rightarrow \8} x_n=\alpha$ and $\lim_{n\rightarrow \8} y_n=\8$. At a cusp $r=\gamma^{-1} \8$, a sequence $z_n$ converges to $\alpha \in \PP^1(\R) \smallsetminus \{ r\}$ if $\gamma z_n$ converges to $\gamma \alpha$.

 The group $\Gamma$ acts on $\bigsqcup_{r \in \PP^1(\Q)}  B_r$ by sending $z \in B_r$ to $\gamma z \in B_{\gamma r}$. Let
\begin{align}
\YBS \coloneqq \Gamma \backslash \overline{\HH}
\end{align}
be the Borel-Serre compactification of $Y$.  If we write $C_N \coloneqq \Gamma \backslash \PP^1(\Q)$ for the set of cusps of $Y$, then
\begin{align}
\YBS = Y \cup  \bigsqcup_{r \in C_N} \Ccal_r
\end{align}
is the union of $Y$ with a circle $\Ccal_r \coloneqq \Gamma_r \backslash B_r \simeq S^1$ at each cusp, where $\Gamma_r$ is the stabilizer of the cusp.

For $r \in \PP^1(\Q)$ and $x \in \HH \cup \PP^1(\R)$ with $x \neq r$, let $\pi_r(x)$ be the endpoint in the horocycle $B_r$ of the geodesic joining $x$ to $r$. Note that this maps is $\Gamma$-equivariant in the sense that $\gamma(\pi_r(x))=\pi_{\gamma r}(\gamma x)$, and continuous: if $ \lim_{n \rightarrow \8} z_n= \alpha \in B_r$, then
\begin{align}
\lim_{n \rightarrow \8} \pi_r (z_n)= \pi_r(\alpha).
\end{align} Given two distinct points $r_1,r_2 \in \HH \cup \PP^1(\R)$, let $\{r_1,r_2\} \in Z_1(\overline{\HH})$ be the geodesic oriented from $r_1$ to $r_2$. If $r_1, r_2$ are in $\PP^1(\Q)$, we call it a {\em modular symbol}. Its boundary in $\overline{\HH}$ is
\begin{align}
\partial \{r_1,r_2\}= \pi_{r_2}(r_1)-\pi_{r_1}(r_2).
\end{align}By abuse of notation, we will also denote by $\{r_1,r_2\}$ its image in $Y$, which represents a $1$-cycle
\begin{align}
\{r_1,r_2\} \in H_1(\YBS,\partial \YBS;\Z).
\end{align}

 \begin{mydef}A modular symbol $\{r_1,r_2\}$ is {\em unimodular} if there is a matrix $\gamma \in \SL_2(\Z)$ such that $\{\alpha,\beta\}=\gamma \{0,\8\}$.
 \end{mydef}

For a cusp $r$ and $x,y \in \HH \cup \PP^1(\R)$ with $x \neq r$, $y \neq r$, we define the {\em modular cap} $[x,y]_r$ to be the segment in $B_r$ from $\pi_r(y)$ to $\pi_r(x)$. We have
\begin{align}
\partial [x,y]_r= \pi_{r}(y)-\pi_{r}(x).
\end{align}
Its image modulo $\Gamma$ represents a $1$-chain in $C_1(\partial \YBS)$. When the modular cap is closed, it represents a homology class which is a multiple of the class $\Ccal_r \in H_1(\partial \YBS)$.

Let $C_k(\YBS,\partial \YBS;\C) \coloneqq C_k(\YBS;\C) /C_k(\partial \YBS;\C) $ be the complex of relative chains, whose homology is the relative homology $H_k(\YBS,\partial \YBS;\C)$. The short exact sequence
\begin{align}
0 \longrightarrow C_k(\partial \YBS;\C) \longrightarrow C_k(\YBS;\C) \longrightarrow C_k(\YBS,\partial \YBS;\C) \longrightarrow 0
\end{align}
induces a long exact sequence
\begin{equation} \label{les1}
 \begin{tikzcd}
H_2(\partial \YBS;\C) \rar[] & H_2( \YBS;\C) \rar[] & H_2(\YBS,\partial \YBS;\C) \ar[overlay, out=-20, in=160]{dll} \\
H_1(\partial \YBS;\C) \rar[] & H_1( \YBS;\C) \rar[] & H_1(\YBS,\partial \YBS;\C) \ar[overlay, out=-20, in=160]{dll} \\
H_0(\partial \YBS;\C) \rar[] & H_0( \YBS;\C) \rar[] & H_0(\YBS,\partial \YBS;\C),
\end{tikzcd}    
\end{equation}
where the connecting morphisms are the boundary operators. Since the inclusion of $Y$ in $\YBS$ is a homotopy equivalence, we have $H_1(Y;\C) \simeq H_1( \YBS;\C)$ and $H_2( \YBS;\C) \simeq H_2( Y;\C) \simeq H_c^0(Y;\C)=0$ (since $Y$ is noncompact). Moreover, note $\partial \YBS$  is a union of circles (one $\Ccal_r$ at each cusp $r$). Thus, we have $H_2(\partial \YBS;\C)=0$ and 
\begin{align}
H_0(\partial \YBS;\C) \simeq H_1(\partial \YBS;\C) \simeq \C[C_N],
\end{align}
where $\C[C_N]$ is the $\C$-vector space generated by the cusps $C_N$. At each cusp, a generator is the closed modular cap $\Ccal_r$, oriented such that the boundary map sends the fundamental class $1 \in H_2(\YBS,\partial \YBS;\C)$ to $\Ccal_r$ at each cusp. We get a long exact sequence
\begin{equation}
 \label{les2}
\begin{tikzcd}
0 \rar[] & 0 \rar[] & \C  \ar[overlay, out=-20, in=160]{dll} \\
\C[C_N] \rar[] & H_1(Y;\C) \rar[] & H_1(\YBS,\partial \YBS;\C)  \ar[overlay, out=-20, in=160]{dll} \\
 \C[C_N] \rar[] & \C \rar[] & H_0(\YBS,\partial \YBS;\C).
\end{tikzcd}    
\end{equation}
The first connecting map $\C \longrightarrow \C[ C_N]$ sends $1$ to $\sum_{r \in C_N} \Ccal_r$. Let $\Ccal(\C)$ be the cokernel of this map, which consists of linear combinations of closed modular caps, modulo the subspace spanned by $\sum_{r \in C_N} \Ccal_r$. We can indentify $\Ccal(\C)$ with a subspace of $\C[C_N]$ as
\begin{align} \label{def ccalC}
\Ccal(\C) \cong \left \{ \left . \sum_{r \in C_N} n_r \Ccal_r \in \C[C_N] \ \right \vert \ n_r \in \C, \ \sum_{r\in C_N} n_r=0 \right\}.
\end{align} In particular, note that its dimension is equal to the number of cusps minus one, {\em i.e.}
\begin{align} \label{dimension modular caps}
\dim_\C(\Ccal(\C)) = \vert C_N \vert - 1.
\end{align} Ash and Rudolph \cite{AR79} showed (in much greater generality for congruence subgroups of $\SL_n$) that the homology $H_1(\YBS,\partial \YBS;\C)$ is generated by modular symbols $\{\alpha, \beta\}$. Hence, the kernel of the second connecting map $H_1(\YBS,\partial \YBS;\C) \longrightarrow  \C[C_N]$ consists of classes represented by linear combinations of modular symbols of degree $0$, {\em i.e.} such that $\partial c =0 \in \C[C_N] \simeq H_0(\partial \YBS;\C)$.  Let us denote by $\Mcal \Scal_0(\C)$ this kernel, so that we get a short exact sequence
\begin{align}
0 \longrightarrow \Ccal(\C) \longrightarrow H_1(Y;\C) \longrightarrow \Mcal \Scal_0(\C) \longrightarrow 0.
\end{align}
Since it is a short exact sequence of vector spaces, the sequence splits and we have an isomorphism
\begin{align} \label{splitting mod caps}
H_1(Y;\C) & \longrightarrow \Ccal(\C)  \oplus \Mcal \Scal_0(\C).
\end{align}

\subsection{Extending forms to the boundary}
\begin{mydef}\label{definition extension} \cite[Definition.~2.1]{S89} We say that a smooth $1$-form on $\HH$ extends to a smooth $1$-form on $\overline{\HH}$ if there is a collection of smooth $1$-forms $\omega^{(r)} \in \Omega^1(B_r)$ at each $r \in \PP^1(\Q)$ such that for each point $P \in B_r$
\begin{align}
\lim_{z \rightarrow P} \omega=\omega^{(r)},
\end{align}
where the limit is over $z$ in $\HH$ approaching $P$.
\end{mydef}
If $\omega$ is closed, $\Gamma$-invariant and extends to $\overline{\HH}$, then it can be viewed as form $\omega \in \Omega^1(\overline{\HH})^\Gamma \cong \Omega^1(\YBS)$. If we write a cycle 
 \begin{align} \label{splitting later}
\Zcal=\sum_{r \in C_N} n_r \Ccal_r + \Mcal \in \Ccal(\C)  \oplus \Mcal \Scal_0(\C)
\end{align}
with respect to the splitting \eqref{splitting mod caps}, then it follows from Stokes theorem that
\begin{align} \label{stokes}
\int_\Zcal \omega= \sum_{r \in C_N} n_r \int_{\Ccal_r} \omega^{(r)} + \int_\Mcal \omega.
\end{align}

 \subsection{Continued fraction algorithm} By work of Manin \cite{M72}, later generalized to $\GL_n$ by Ash–Rudolph \cite{AR79}, every modular symbol is homologous to a linear combination of unimodular symbols, via an explicit algorithm using continued fractions. Stevens \cite[Section~1.8]{S89} extends Manin's algorithm to incorporate modular caps. We will now explain how to modify Steven's algorithm to explicitly decompose a cycle $\Zcal_\gamma = \{z_0, \gamma z_0\}$ as in \eqref{splitting later}, and write it as a linear combination of unimodular symbols and closed modular caps.

If $\gamma$ is a parabolic matrix that stabilizes the cusp $r \in \PP^1(\Q)$, then
\begin{align}
\Zcal_\gamma = \{z_0,\gamma z_0 \} & =\left [z_0,\gamma z_0 \right ]_r;
\end{align}
see Figure \ref{moving to boundary 2}. \begin{figure}[h] 
\centering
\captionsetup{width=.70\textwidth,font={small,it}}
\includegraphics[scale=0.35]{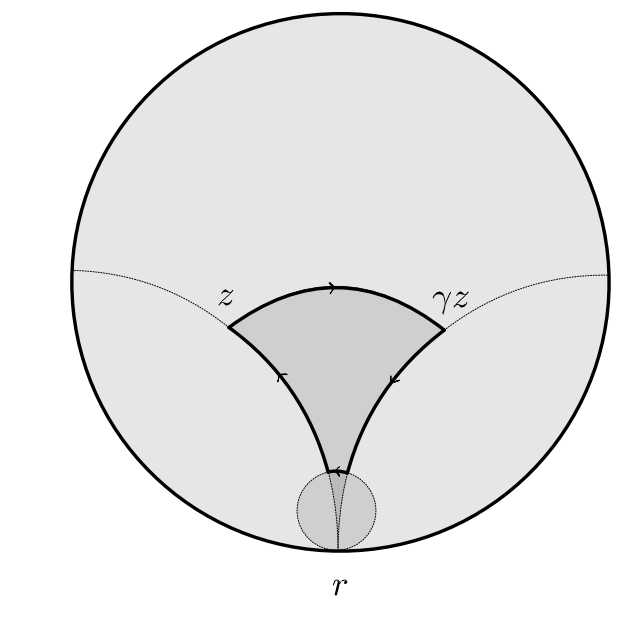}
\caption{We visualize the hyperbolic $2$-space in the disk model. The cycle $\{z,\gamma z\}$ is moved to the boundary component at the cusp $r$, represented by the arc on the horocycle. The two sides are $\Gamma$-translates and cancel out.}
\label{moving to boundary 2}
\end{figure} Let $\gamma_r \in \SL_2(\Z)$ be such that $\gamma_r\8=r$. The matrix $\gamma_r^{-1} \gamma \gamma_r \in \SL_2(\Z)$ preserves $\8$ and 
\begin{align} \label{def bgamma}
\gamma_r^{-1} \gamma \gamma_r=\begin{pmatrix}
1 & b(\gamma) \\ 0 & 1
\end{pmatrix}
\end{align}
for some $b(\gamma) \in \Z$.
\begin{rmk}
If $w_r$ denotes the width of the cusp, then $w(\gamma) \coloneqq \frac{b(\gamma)}{w_r}$ is the winding number of the loop around the cusp. 
\end{rmk}
If we set  $z'_0 \coloneqq \gamma_r^{-1}z_0$, we find that
\begin{align}
[z_0,\gamma z_0 ]_{r}=\gamma_r [\gamma_r^{-1}z_0,\gamma_r^{-1} \gamma z_0 ]_{\8}=\gamma_r [z'_0,\gamma_r^{-1} \gamma \gamma_r z'_0 ]_{\8}=b(\gamma)\gamma_r[0,1]_\8=b(\gamma)\Ccal_r.
\end{align}

Let $\gamma \in \Gamma$ be a matrix such that $c \neq 0$. Consider the continued fraction of $\frac{a}{c}=\gamma \8$
\begin{align}\label{continued fraction}
\frac{a}{c}= b_0-\cfrac{1}{b_1-\cfrac{1}{\cdots-\cfrac{1}{b_{n-1}-\cfrac{1}{b_n}}}}
\end{align}
with $b_0 \in \Z$, $b_k \in \NN$ for $k \geq 1$, and convergents $\frac{p_0}{q_0}=\frac{b_0}{1}, \frac{p_1}{q_1},\dots ,\frac{p_n}{q_n}=\frac{a}{c}$. For $-1 \leq k \leq n$ let $\gamma_k \in \SL_2(\Z)$ be the matrices $\gamma_{-1} \coloneqq \begin{pmatrix} 1 & 0 \\ 0 & 1 \end{pmatrix}$ and 
\begin{align}
\gamma_k \coloneqq \begin{pmatrix}
-p_k & p_{k-1} \\ -q_k & q_{k-1}
\end{pmatrix} =\begin{pmatrix}
-b_0 & 1 \\ -1 & 0
\end{pmatrix} \cdots \begin{pmatrix}
-b_k & 1 \\ -1 & 0
\end{pmatrix}.
\end{align}
In particular, we have $b_0=\frac{p_0}{q_0}$ and $\frac{a}{c}=\frac{p_n}{q_n}$, as well as the recursion
\begin{align}
p_{k+1} \coloneqq p_{k-1}-b_{k+1}p_k, \quad q_{k+1} \coloneqq q_{k-1}-b_{k+1}q_k.
\end{align}
For $k \geq 0$ we deduce the following action of $\gamma_k$ on the cusps
\begin{align}
\gamma_k b_{k+1}=\frac{p_{k+1}}{q_{k+1}}, \ \ \  \gamma_k \8=\frac{p_k}{q_k}, \ \ \   \gamma_k 0 =\frac{p_{k-1}}{q_{k-1}}.
\end{align}

Let $\Pcal$ be the closure of the hyperbolic polygon with endpoints the cusps $p_0/q_0, \dots, p_n/q_n,\8$ as well as the interior points $z_0$ and $\gamma z_0$, which is closed by adding modular caps at each cusp; see Figure \ref{moving to boundary bis}.

 \begin{figure}[h] 
\centering
\captionsetup{width=.75\textwidth,font={small,it}}
\includegraphics[scale=0.30]{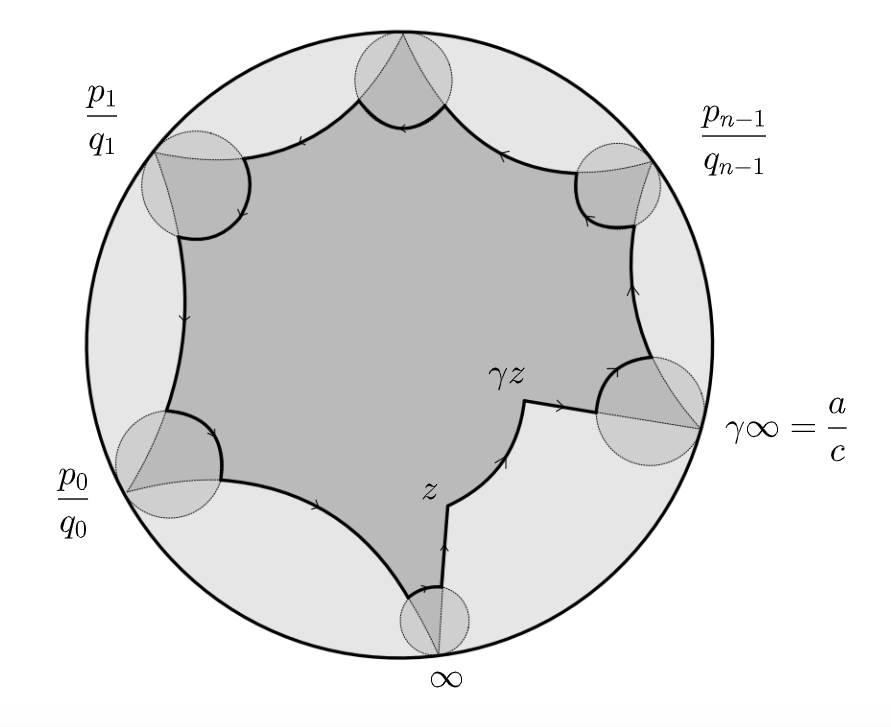}
\caption{The polygon $\Pcal$.}
\label{moving to boundary bis}
\end{figure}

 Its boundary consists of modular caps, modular symbols and some geodesic segments with endpoints in the interior of $\HH$:
\begin{align}
\partial \Pcal= & \left [\frac{p_0}{q_0}, z_0 \right ]_{\8} + \left [\gamma z_0, \frac{p_{n-1}}{q_{n-1}} \right ]_{\gamma \8}+ \sum_{k=1}^{n-1} \left [\frac{p_{k+1}}{q_{k+1}}, \frac{p_{k-1}}{q_{k-1}} \right ]_{\frac{p_k}{q_k}}+\left [\frac{p_1}{q_1}, \8 \right ]_{\frac{p_0}{q_0}} \hspace{0.8cm} \textit{(modular caps)}\\
& +  \sum_{k=1}^n \left \{ \frac{p_{k}}{q_{k}}, \frac{p_{k-1}}{q_{k-1}} \right \} + \left \{ \frac{p_0}{q_0}, \8 \right \} \hspace{5.7cm} \textit{(modular symbols)}\\
& + \{ z_0,\gamma z_0\}+ \{ \8, z_0\}+\{ \gamma z_0,\gamma \8 \}. \hspace{3.5cm} \textit{(remaining geodesic segments)}
\end{align}

It follows from the previous relations that
\begin{align}
\left [\frac{p_{k+1}}{q_{k+1}}, \frac{p_{k-1}}{q_{k-1}} \right ]_{\frac{p_k}{q_k}} & = \gamma_k [b_{k+1},0]_\8, \ \ 1 \leq k \leq n-1 \\ 
\left [\frac{p_1}{q_1}, \8 \right ]_{\frac{p_0}{q_0}}  & = \gamma_0 \left [b_1, 0 \right ]_{\8},\\
\left \{ \frac{p_k}{q_k}, \frac{p_{k-1}}{q_{k-1}} \right \} & = \gamma_k \{ \8,0\}, \ \ \ 1 \leq k \leq n\\
\left \{ \frac{p_0}{q_0}, \8 \right \} & = \gamma_0 \{ \8,0\}, \\
\left [\gamma z_0, \frac{p_{n-1}}{q_{n-1}} \right ]_{\gamma \8} & =\left [\gamma z_0, \frac{p_{n-1}}{q_{n-1}} \right ]_{\gamma \8}.
\end{align}
Moreover, we have $\gamma \8=\frac{a}{c}=\frac{p_n}{q_n}=\gamma_n \8$, so that 
\begin{align}
\gamma_n=\gamma \begin{pmatrix} 1 & m \\ 0 & 1\end{pmatrix}, 
\end{align}
where $m=p_{n-1}d-bq_{n-1}$. Thus, we have $\frac{p_{n-1}}{q_{n-1}}=\gamma_n 0 = \gamma m$ and 
\begin{align}
\left [\gamma z_0, \frac{p_{n-1}}{q_{n-1}} \right ]_{\gamma \8}=\gamma [z_0, m ]_{\8} = [z_0, m ]_{\8},
\end{align}
where the last step follows from the fact that $\gamma \in \Gamma$. Similarly, we have $\{ \8, z_0\}+\{ \gamma z_0,\gamma \8 \}=0$ in the homology of $Y$.
It follows that
\begin{align}
\partial \Pcal= &\left [b_0,z_0 \right ]_{\8} + \left [z_0, m \right ]_{\gamma \8}- \sum_{k=1}^{n} b_k\gamma_{k-1} [0,1]_\8 +  \sum_{k=1}^n \gamma_k \{ \8,0\} + \left \{ b_0, \8 \right \} + \{ z_0,\gamma z_0\},
\end{align}
from which we deduce the following theorem.

\begin{thm} \label{split cycle} Under the isomorphism $H_1(Y;\C) \simeq \Ccal(\C)  \oplus \Mcal \Scal_0(\C)$ we can write the cycle as
\begin{align}
\Zcal_\gamma  = \begin{cases} \ 
  b(\gamma) \Ccal_r & \ \ \textrm{if} \ \gamma \ \textrm{is parabolic}, \\[1em]
 \ (b_0+bq_{n-1}-p_{n-1}d)\Ccal_\8 + \sum_{k=0}^{n-1} b_{k+1} \Ccal_{\gamma_k\8} +  \sum_{k=0}^n \gamma_k \{ 0,\8\} & \ \ \textrm{if} \ \gamma \ \textrm{is hyperbolic}.
\end{cases}
\end{align}
\end{thm}
Note that $\sum_{k=0}^n \gamma_k \{ 0,\8\}$ lies in $\Mcal \Scal_0(\C)$ since
\begin{align}
\partial \sum_{k=0}^n \gamma_k \{ 0,\8\}= \sum_{k=1}^n \left ( \frac{p_k}{q_k} -  \frac{p_{k-1}}{q_{k-1}} \right ) + \left ( \frac{p_0}{q_0}-\8 \right )= \frac{p_n}{q_n}-\8=0,
\end{align}
where the last step follows from the fact that $\frac{p_n}{q_n}=\frac{a}{c}=\gamma \8 \equiv  \8$ modulo $\Gamma$.

 \subsection{Hecke operators} \label{section Hecke} Let us now define the Hecke operators for $\Gamma=\Gamma_1(N)$, following \cite[p.~111]{langmodular}. We define
\begin{align}
    \Delta^{(n)}_1(N)= \left \{ \left . M=\begin{pmatrix}
       a & b \\ Nc & d 
    \end{pmatrix} \in \Mat_2(\Z) \ \right \vert \ \det(M)=n, \ \ a \equiv 1 \mod N \right \}
\end{align}
for $n>0$.
The congruence subgroup $\Gamma_1(N) \subset \SL_2(\Z)$ acts by multiplication on the left and the right of $\Delta^{(n)}_1(N)$, and the quotient $\Gamma_1(N)  \backslash \Delta^{(n)}_1(N)$ is finite. An explicit choice of representatives is given by matrices
\begin{align} \label{doublecosets}
\gamma(a,b,d)=\sigma_a \begin{pmatrix} 
      a & b \\ 0 & d
  \end{pmatrix}, \qquad ad=n, \ \ d>0, \ \ (a,N)=1, \ \ b=0, \dots,d-1,
\end{align}
where for each $a$ dividing $n$ and coprime to $N$ we choose a matrix $\sigma_a$ in $\SL_2(\Z)$ such that
\begin{align}
\sigma_a \equiv \begin{pmatrix}
a^{-1} & 0 \\ 0 & a 
\end{pmatrix} \ \mod{N}.
\end{align}
For $\beta \in \GL_2(\Q)^+$ and $k \in \Z$ we have the usual slash operator $f[\beta]_k=\det(\beta)^{k-1}j(\beta,\tau)^{-2} f(\beta \tau)$. The Hecke operators act on a modular form $f \in M_2(\Gamma_1(N))$ by
\begin{align}
T_nf \coloneqq \sum_{\gamma \in \Gamma_1(N)  \backslash \Delta^{(n)}_1(N)} f[\gamma]_2.
\end{align} 
The action of the coset representatives also induces a Hecke operator
\begin{align}
    T_n\{\alpha,\beta\} \coloneqq \sum_{\gamma \in \Gamma_1(N)  \backslash \Delta^{(n)}_1(N)}  \{\gamma \alpha,\gamma \beta\}
\end{align}
on modular symbols. Finally, the Hecke operators also act on differentials forms by sending $\omega \in \Omega^1(Y_1(N))\cong \Omega^1(\HH)^{\Gamma_1(N)}$ to
\begin{align}
T_n \omega \coloneqq \sum_{\gamma \in \Gamma_1(N)  \backslash \Delta^{(n)}_1(N)} \gamma^\ast \omega.
\end{align}
It induces an action of Hecke operators on the cohomology group $H^1(Y_1(N);\C)$. If $\omega_{\{\alpha,\beta\}} \in \Omega^1(Y_1(N))$ is a Poincaré dual to  $\{\alpha,\beta\}$, then $T_n\omega_{\{\alpha,\beta\}}=\omega_{T_n\{\alpha,\beta\}}$. 
If $\omega_{f}=f(z)dz \in \Omega^1(Y_1(N))$ is the $1$-form associated with a modular form $f \in M_2(\Gamma_1(N))$, then $T_n\omega_{f}=\omega_{T_n f}$.

\subsection{Modular forms and Kronecker-Eisenstein series}  Let $e(\alpha)=e^{2i\pi \alpha}$. We start by introducing the Kronecker-Eisenstein series
\begin{align}
   \Kcal_k(s,\tau,\lambda,\mu) \coloneqq \left ( \frac{y}{\pi} \right )^{s-k} \frac{\Gamma(s)}{(-2i\pi)^k} \sideset{}{^\prime}\sum_{\omega \in \Z+\tau\Z} \frac{\overline{\omega+\lambda}^k}{\vert \omega+\lambda \vert^{2s}} e \left (\frac{\im(\omega \bar{\mu})}{y} \right ), \qquad \tau=x+iy,
\end{align}
defined for a non-negative integer $k$, and complex numbers $\lambda,\mu,s$. It converges for $\re(s)>1+\frac{k}{2}$ and the $'$ means that we remove $w=-x$ from the summation if $x$ is in $\Z \tau+\Z$. This is the series considered by Weil in \cite[section~VIII]{weilelliptic}. The function admits a meromorphic continuation to the whole plane with only pole at $s=1$ (if $k=0$ and $\lambda$ is in $\Z \tau+\Z$); see \cite[section~VIII, p.~80]{weilelliptic}.  Moreover, it satisfies the functional equation
\begin{align} \label{funceq2}
    \Kcal_k(s,\tau,\lambda,\mu)=e \left (\frac{\im(\bar{\lambda}\mu)}{y} \right ) \Kcal_k(1+k-s,\tau,\mu,\lambda).
\end{align}

We will consider the case where $\lambda,\mu$ are $N$-torsion points on $\C/\Z\tau+\Z$. For two integers $p,q \in \Z$, we define
\begin{align} \label{defeiseE}
E^{(k,l)}_{p,q}(\tau) & \coloneqq \Kcal_{k+l} \left (k,\tau,0,\frac{p\tau+q}{N} \right ) \\
&= \frac{(k-1)!}{(-2i\pi)^{k+l}} \left ( \frac{y}{\pi} \right )^{-l} \  \lim_{s \rightarrow 0} \sideset{}{^\prime}\sum_{m,n \in\Z} \frac{\overline{m\tau+n}^l}{(m\tau+n)^k \vert m\tau+n \vert^{2s}}e\left ( \frac{mq-np}{N}\right )
\end{align}
and
\begin{align}
\Ehat^{(k,l)}_{p,q}(\tau) & \coloneqq \Kcal_{k+l} \left (k,\tau,\frac{p\tau+q}{N},0 \right ) \\
&= N^{k-l} \frac{(k-1)!}{(-2i\pi)^{k+l}} \left ( \frac{y}{\pi} \right )^{-l} \  \lim_{s \rightarrow 0} \sideset{}{^\prime}\sum_{\substack{m,n \in\Z \\ m \equiv p \mod{N} \\ n \equiv q \mod{N}}} \frac{\overline{m\tau+n}^l}{(m\tau+n)^k \vert m\tau+n \vert^{2s}}.
\end{align}
For these forms, the functional equation gives
\begin{align} \label{functional equation Gkl}
\Ehat^{(k,l)}_{p,q}(\tau)=E^{(1+l,k-1)}_{p,q}(\tau).
\end{align}
In particular, when $l=0$ we get the Eisenstein series
\begin{align}
\Ehat^{(k)}_{p,q}(\tau) \coloneqq \Ehat^{(k,0)}_{p,q}(\tau), \qquad E^{(k)}_{p,q}(\tau) \coloneqq E^{(k,0)}_{p,q}(\tau). 
\end{align}
They define holomorphic modular forms in $M_k(\Gamma(N))$ in weight $k \neq 2$. When $k=2$, the Eisenstein series $E^{(2)}_{p,q}(\tau)$ with $(p,q) \not \equiv (0,0) \mod{N}$ are holomorphic, whereas $\Ehat^{(2)}_{p,q}$ transforms like a modular form of weight $2$ but is non-holomorphic.

We have an orthogonal decomposition $M_k(\Gamma(N)) \coloneqq  E_k(\Gamma(N)) \oplus S_k(\Gamma(N))$, where the Eisenstein spaces in weight $1$ and $2$ (the only weights that we will consider) are
\begin{align}
E_1(\Gamma(N)) &\coloneqq  \Span \left \{  \left . \Ehat^{(1)}_{p,q}(\tau) \right \vert \ (p,q) \in (\Z/N\Z)^2 \right \}, \\
E_2(\Gamma(N))& \coloneqq  \Span \left \{  \left . \Ehat^{(2)}_{p,q}(\tau) \right \vert \ (p,q) \in (\Z/N\Z)^2 \right \} \cap M_2(\Gamma(N))
\end{align}
see \cite[Theorem.~7.2.18]{M89}. Since only the constant term of $\Ehat^{(2)}_{p,q}(\tau)$ is non-holomorphic, we can also characterize the Eisenstein space in weight $k=2$ as
\begin{align}
E_2(\Gamma(N)) \coloneqq  \left \{  \left . \sum_{(p,q) \in (\Z/N\Z)^2 }a_{pq} \Ehat^{(2)}_{p,q}(\tau) \ \right \vert \  \sum_{(p,q) \in (\Z/N\Z)^2 }a_{pq} =0 \right \};
\end{align}
see \cite[Section.~5.11]{DS05}. In particular, the difference $\Ehat^{(2)}_{p,q}-\Ehat^{(2)}_{0,0}$ is holomorphic. For a congruence subgroup $\Gamma \supseteq \Gamma(N)$ we set $E_k(\Gamma) \coloneqq  E_k(\Gamma(N)) \cap M_k(\Gamma)$, and we have a decomposition
\begin{align}
M_k(\Gamma)=E_k(\Gamma) \oplus S_k(\Gamma),
\end{align}
which is orthogonal with respect to the Petersson inner product. By the dimension formulas for modular forms, we have
\begin{align} \label{dimension eisenstein}
\dim_\C(E_2(\Gamma_1(N))=\vert C_N \vert-1,
\end{align}
where $\vert C_N \vert$ is the number of cusps; see \cite[Formula.~4.3]{DS05}.

Let us also consider the special values
\begin{align} \label{def g kronecker}
G_r^{(k)}(\tau) \coloneqq E^{(k)}_{r,0}(N\tau), \qquad \Ghat_r^{(k)}(\tau) \coloneqq \Ehat^{(k)}_{r,0}(N\tau).
\end{align}
Explicitly, they are given by the sums
\begin{align} \label{def g series}
G_r^{(k)}(\tau) & = N \frac{(k-1)!}{(-2i\pi)^k}  \lim_{s \rightarrow 0} \sideset{}{^\prime}\sum_{m,n \in\Z} \frac{1}{(mN\tau+n)^k \vert mN\tau+n \vert^{2s}}e\left (- \frac{rn}{N}\right ), \\
\Ghat_r^{(k)}(\tau) & = \frac{(k-1)!}{(-2i\pi)^k}   \lim_{s \rightarrow 0} \sideset{}{^\prime}\sum_{\substack{m,n \in\Z \\ m \equiv r \mod{N}}} \frac{1}{(m\tau+n)^k \vert m\tau+n \vert^{2s}}.
\end{align}
They are again holomorphic, except in some cases in weight $k=2$ as above. Note that for $k=1$, we have $G_r^{(1)}(\tau)=\Ghat_r^{(1)}(\tau)$ by the functional equation. Finally, let us also define
\begin{align} \label{Hpq def}
H_{p,q}^{(2)}(\tau) \coloneqq G^{(2)}_{q}(\tau)-\delta_{q0}\Ghat^{(2)}_{p}(\tau)
\end{align}
where we will use the notation
\begin{align}\delta_{q0}= 
\begin{cases}
1 & \textrm{if}  \ q \equiv 0 \mod N \\
0 & \textrm{otherwise}.
\end{cases}
\end{align}
\begin{lem} \label{lemma eisenstein} For $k=1$, the modular forms $G_r^{(1)}(\tau)$ lie in the Eisenstein space $E_1(\Gamma_1(N))$. For $k=2$, the modular form $H_{p,q}^{(2)}(\tau)$ lies in the Eisenstein space $E_2(\Gamma_1(N))$. 
\end{lem} 
\begin{proof} Let $\zeta_N=e\left (\frac{1}{N} \right )$. Notice that
\begin{align}
G_r^{(k)}(\tau) = \sum_{l \in \Z/N\Z} \zeta_N^{-rl}\Ehat^{(k)}_{0,l}(\tau), \qquad \Ghat_r^{(k)}(\tau)= \sum_{l \in \Z/N\Z}\Ehat^{(k)}_{r,l}(\tau), 
\end{align}
and that both lie in $M_k(\Gamma_1(N))$. For $k=1$, they lie in $E_1(\Gamma_1(N))$, by definition. For $k=2$, the modular form $H_{p,q}^{(2)}$ is in $E_2(\Gamma_1(N))$ since
\begin{align}
\sum_{l \in \Z/N\Z}\zeta_N^{-ql}-\delta_{q0}N=0.
\end{align}
\end{proof}

\subsection{Fourier expansions of Eisenstein series}  
\begin{prop}\cite[Lemma.~3.3]{B17} Let $k \geq 1$ and $(p,q)\not \equiv (0,0) \mod{N}$. We have the Fourier expansion
\begin{align}
E^{(k)}_{p,q}(\tau)=a_0(E^{(k)}_{p,q})+N^{1-k} \left (\sum_{\substack{ m,n \geq 1 \\ n \equiv p (N)}}\zeta_N^{mq}n^{k-1}e \left ( \frac{mn}{N}\tau\right )+(-1)^k\sum_{\substack{ m,n \geq 1 \\ n \equiv -p (N)}}\zeta_N^{-mq}n^{k-1}e \left ( \frac{mn}{N}\tau\right ) \right )
\end{align}
where the constant term is
\begin{align}
a_0(E^{(1)}_{p,q})= \begin{cases} \ \frac{1}{2}\frac{1+\zeta^q_N}{1-\zeta_N^q}& \textrm{if} \ p=0, \ q \neq 0 \\[5mm]
\ -B_1\left (\left \{ \frac{p}{N} \right \} \right )& \textrm{if} \ p \neq 0
\end{cases}
\end{align}
in weight $k=1$, and
\begin{align}
a_0(E^{(2)}_{p,q})=-\frac{1}{2}B_2\left (\left \{ \frac{p}{N} \right \} \right )
\end{align}
in weight $k = 2$, where $B_k(t)$ is the Bernoulli polynomial and $\{ t \}$ denotes the fractional part of $t$.
\end{prop}
\begin{rmk} \label{remark comparison} By comparing with the Fourier expansion of \cite[Proposition.~3.6]{BG01}, we see that
\begin{align}
s_{a/N}\left ( \tau \right )=-E^{(1)}_{0,a}(\tau)=-\Ehat^{(1)}_{0,a}(\tau),
\end{align}
where $s_{a/N}$ is the weight one Eisenstein series that appears in \cite[Theorem.~4.11]{BG01}. By comparing with the formulas in \cite[Proposition.~3.8]{BG01} and \cite[164]{B74}, we find that the weight $2$ Eisenstein series $s^{(2)}_{a/N}$ and $s^2_{a/N}$ considered by Borisov-Gunnells is related to ours by
\begin{align}
s^{(2)}_{a/N}(\tau)& =\Ehat^{(2)}_{0,0}(\tau)-\Ehat^{(2)}_{0,a}(\tau), \\
s^{(2)}_{a/N}(\tau)+s^{2}_{a/N}(\tau)& =-2\sum_{p,q \mod N} \zeta_N^{ap} \Ehat^{(2)}_{p,q}(\tau)=-2\sum_{p \mod N} \zeta_N^{ap}\Ghat^{(2)}_{p}(\tau).
\end{align}
\end{rmk}

In Theorem \ref{corollary intro} we stated that
\begin{align} \label{relationeisenstein}
G_a^{(1)}G_{b}^{(1)}+G_b^{(1)}G_{c}^{(1)}+G_c^{(1)}G_{a}^{(1)}=G_a^{(2)}+G_b^{(2)}+G_c^{(2)}
\end{align}
for three integers $a,b,c$ that satisfy $a+b+c \equiv 0 \mod N$ and such that $a,b,c \neq 0 \mod N$. As a consistency check, we verify the constant terms on both sides. Since $G_r^{(k)}(\tau) =E^{(k)}_{r,0}(N\tau)$, the two Eisenstein series have the same constant term
\begin{align}
a_0(G_r^{(2)})&=a_0(E^{(2)}_{r,0})=-\frac{1}{2}B_2\left (\left \{ \frac{r}{N} \right \} \right ), \\
a_0(G_a^{(1)}G_{b}^{(1)})&=B_1\left (\left \{ \frac{a}{N} \right \} \right )B_1\left (\left \{ \frac{b}{N} \right \} \right ).
\end{align}
Finally, the condition $a+b+c \equiv 0 \mod N$ implies that the numbers $x=\{ \frac{a}{N} \}$, $y=\{\frac{b}{N}\}$ and $z=\{\frac{c}{N}\}$ satisfy $x+y+z=1$ or $2$. Finally, taking the constant term of \eqref{relationeisenstein} is consistent with the following relation between Bernoulli polynomials.

\begin{lem} Let $x,y,z$ be three real numbers such that $x+y+z=1$ or $2$ . Then
\begin{align}
B_1(x)B_1(y)+B_1(y)B_1(z)+B_1(z)B_1(x)=-\frac{1}{2} \left ( B_2(x)+B_2(y)+B_2(z)\right ).
\end{align}
\end{lem}
\begin{proof}
We have $B_1(t)=t-\frac{1}{2}$ and $B_2(t)=t^2-t+\frac{1}{6}=B_1(t)^2-\frac{1}{12}$. For three real numbers $x,y,z$, we have
\begin{align}
\left (B_1(x)+B_1(y)+B_1(z)\right )^2= & B_2(x)+B_2(y)+B_2(z)+\frac{1}{4} \\
& +2\left ( B_1(x)B_1(y)+B_1(y)B_1(z)+B_1(z)B_1(x) \right ).
\end{align}
Thus, if $x+y+z=1$ or $2$, then $B_1(x)+B_1(y)+B_1(z)= \pm \frac{1}{2}$ and the equality follows.
\end{proof}

\section{Theta lift}
We recall the construction of \cite{rbrsln} in the setting of the modular curve.
\subsection{Symmetric space of $\GL_2(\R)$} Let $X$ be the set of positive-definite quadratic forms on $\R^2$, which can also be identified with the space of positive-definite symmetric $2 \times 2$ matrices. 
The group $\GL_2(\R)^+$ acts on a symmetric matrix $M$ in $X$ by $M \longmapsto gMg^T$, and the stabilizer of the symmetric matrix $M=\id_2$ is $\SO(2)$. Let $A_\R \simeq \R_{>0}$ be the center in $\GL_2(\R)^+$. We have a natural diffeomorphism
 \begin{align}
     \HH \times A_\R \simeq \GL_2(\R)^+/\SO(2) \simeq X,\qquad
     (z , t) \longmapsto zt=\frac{t^2}{v} \begin{pmatrix}
             \vert z \vert^2 & u \\ u & 1
         \end{pmatrix}.
 \end{align}
 Let $z=u+iv$ be the coordinates on $\HH$. The first map sends $(z ,t) \in \HH \times A_\R $ to $tg_z\SO(2)$ where
 \begin{align}
     g_z=\begin{pmatrix}
         \sqrt{v} & u/\sqrt{v} \\
         0 & 1/\sqrt{v}
     \end{pmatrix}.
 \end{align}
 The second map sends it to the positive definite quadratic form
 \begin{align}
         t^2g_z g_z^T=\frac{t^2}{v} \begin{pmatrix}
             \vert z \vert^2 & u \\ u & 1
         \end{pmatrix} \in X.
 \end{align}

\subsection{Tautological bundle} \label{tautological bundle}
Let $V$ be the vector space $\Mat_2(\Q)$. We identify $V \simeq \Q^2 \oplus \Q^2$ by writing a matrix $\vbf=[m,n]$, where $m=\begin{bsmallmatrix} m_1 \\ m_2 \end{bsmallmatrix}$ and $n=\begin{bsmallmatrix} n_1 \\ n_2 \end{bsmallmatrix}$ are the columns of $\vbf$. Let $\langle m,n\rangle=m^Tn$ be the standard inner product on $\R^2$, and let $Q$ be the bilinear form on $V$ defined by
\begin{align}
Q(\vbf,\vbf') \coloneqq \langle m,n' \rangle+\langle m',n \rangle, \qquad \vbf=[m,n], \quad \vbf'=[m',n'].
\end{align} The associated quadratic form
\begin{align}
Q(\vbf) \coloneqq  \frac{1}{2}Q(\vbf,\vbf)=\langle m,n \rangle
\end{align} 
is of signature $(2,2)$. The group $\GL_2(\R)$ acts on $V$ by
\begin{align} \label{define rho}
\rho_g(\vbf)=[gm,g^{-t}n], \qquad \vbf=[m,n].
\end{align}
This defines a representation $\rho \colon \GL_2(\Q) \longrightarrow \SO(V) \simeq \SO(2,2)(\Q)$. 
Over $X= \HH \times A_\R$ we have a rank $2$-bundle
\begin{align} \label{bundle B defined}
B=\GL_2(\R)^+ \times_{\SO(2)} \R^2,
\end{align}
consisting of pairs $(g,v) \in \GL_2(\R)^+ \times \R^2$, modulo the equivalence relation $(gk,k^{-1}v)=(g,v)$ for $k \in \SO(2)$. For $\vbf=[m,n]$ we define a section $s_\vbf \colon \HH \times A_\R \longrightarrow B$ by
\begin{align}
s_\vbf(z,t)=\left ( tg_z, \frac{t^{-1}g_z^{-1}m-tg_z^Tn}{2} \right ).
\end{align}
We denote by $X_\vbf$ the zero locus of this section, and by $S_\vbf$ its projection onto $\HH$. 

\begin{mydef}A vector $\vbf=[m,n]$ in $V_\R$ is {\em regular} if the columns $m,n$ are both nonzero vectors, and {\em singular} otherwise. We say that $\vbf$ is positive if $Q(\vbf)>0$ (and negative if $Q(\vbf)<0$).
\end{mydef}
\begin{prop} \label{properties of Xv}We have the following.
\begin{enumerate}
\item The locus $X_\vbf$ satisfies the equivariance $gX_\vbf=X_{\rho(g)\vbf}$ and $gS_\vbf=S_{\rho(g)\vbf}$ for all $g \in \SL_2(\R)$.
\item If $\vbf$ is a positive vector, then the locus $X_\vbf$ is a submanifold of codimension $1$ in $X$. The restriction of the projection $\HH \times A_\R \longrightarrow \HH$ is a diffeomorphism onto $X_\vbf \longrightarrow S_\vbf$. Moreover, if $\vbf=\begin{bsmallmatrix} m_1 & n_1 \\ m_2 & n_2\end{bsmallmatrix}$, then $S_\vbf$ is the modular symbol
\begin{align}
S_\vbf = \left \{\frac{n_2}{n_1},-\frac{m_1}{m_2} \right \}.
\end{align}
 \item If $\vbf=\begin{bsmallmatrix} 0 & 0 \\ 0 & 0\end{bsmallmatrix}$ then $X_\vbf=X$.
 \item In all other cases ($\vbf$ nonzero singular or nonpositive regular), we have $X_\vbf=\emptyset$.
 \end{enumerate}
\end{prop}
\begin{proof} The fact that $S_\vbf$ is the modular $\left \{\frac{n_2}{n_1},-\frac{m_1}{m_2} \right \}$ was computed in \cite[Proposition.~5.1]{rbrsln}. The other points were proved in \cite[Proposition.~2.1]{rbrsln} and \cite[Proposition.~2.2]{rbrsln} for the more general modular symbols considered in {\em loc. cit.} .
\end{proof}
\subsection{Special cycles}

We start by defining some lattice cosets. Let $V_\Z \coloneqq \Mat_2(\Z)$. For a matrix $\xbf_0 \in \Mat_2(\Z/N\Z)$, let 
\begin{align}
L_{\xbf_0} \coloneqq \xbf_0+NV_\Z.
\end{align}
More generally, we denote by $L \subseteq V_\Z$ any linear combination of the form
\begin{align} \label{linear combination}
L = \sum_{\xbf_0 \in \Mat_2(\Z/N\Z)}n_{\xbf_0}L_{\xbf_0}
\end{align}
with $n_{\xbf_0} \in \Z$. We will mainly use the lattice coset
\begin{align} \label{lattice def}
L_{p,q} \coloneqq \left \{ \left . \begin{bmatrix} m_1 & n_1 \\ m_2 & n_2\end{bmatrix} \in \Mat_2(\Z) \ \right \vert \ \begin{bmatrix} m_1  \\ m_2 \end{bmatrix} \equiv \begin{bmatrix} p  \\ q \end{bmatrix} \mod N \right \},
\end{align}
defined for two integers $(p,q) \in \Z^2$. It is the linear combination
\begin{align}
L_{p,q}=\sum_{k,l \in \Z/N\Z } L_{\begin{bsmallmatrix} p & k \\ q & l \end{bsmallmatrix}}.
\end{align}
Let $\Gamma \subseteq \SL_2(\Z)$ be a subgroup stabilizing $L$ under $\rho$ {\em i.e.} $\rho_\gamma L= L$ for all $\gamma \in \Gamma$. We will take 
\begin{align} \Gamma =
\begin{cases} \Gamma_1(N) \ & \ \textrm{if} \ L=L_{p,q} \ \textrm{and} \ (p,q) \equiv (1,0) \mod N, \\
\Gamma(N) \ & \ \textrm{otherwise}.
\end{cases}
\end{align}
Since $\Gamma$ is a subgroup of $\SL_2(\R)$, it does not act on the determinant of a matrix in $X$. Hence, the quotient
\begin{align}
\Gamma \backslash X \xrightarrow{\sim} (\Gamma \backslash \HH) \times A_\R,
\end{align}
is a bundle over $Y=\Gamma \backslash \HH$. By the equivariance property, the image of $X_\vbf \subset X$ in $\Gamma \backslash X$ only depends on the class $[\vbf] \in \Gamma \backslash L$, where $\Gamma$ acts on $L$ by $\rho$. The same holds for $S_\vbf$, and we denote by $\Zcal_{[\vbf]}$ the image of $S_\vbf$ in $Y$. It represents a homology class
\begin{align}
\Zcal_{[\vbf]} \in  H_1(\YBS,\partial \YBS;\Z).
\end{align}
\begin{mydef} For a positive integer $n$, we define the {\em special cycles}
\begin{align} \label{definition of SL}
\Zcal_n(L) \coloneqq \sum_{\substack{[\vbf] \in \Gamma \backslash L \\ Q(\vbf)=n}} \Zcal_{[\vbf]} \in H_1^{\BM}(Y;\Z).
\end{align}
\end{mydef}
\begin{prop} \label{hecke translate} We have 
\begin{align}
\Zcal_n(L_{1,0})=T_n\{0,\8\} \in H_1(\YBSN,\partial \YBSN;\Z).
\end{align}
\end{prop}
\begin{proof} The proof is a minor modification of the proof in \cite[Proposition.~5.1]{rbrsln}, where we used the Hecke operators for $\Gamma_0(N)$ instead of $\Gamma_1(N)$. We include it for the convenience of the reader. 

Let $L_{1,0}^{(n)}$ be the set of vectors in $L_{1,0}$ of determinant $n$. This set coincides with $\Delta^{(n)}_1(N)$ that was used to define the Hecke operators in Section \ref{section Hecke}, but with a different action of $\Gamma_1(N)$. On $\Delta^{(n)}_1(N)$ it acts by left matrix multiplication, whereas on $L_{1,0}^{(n)}$ it acts by
\begin{align}
\rho_\gamma \begin{bmatrix}
m_1 & n_1 \\ m_2 & n_2
\end{bmatrix}= \begin{bmatrix}
\gamma \begin{pmatrix}
m_1 \\ m_2
\end{pmatrix},\gamma^{-T} \begin{pmatrix}
n_1 \\ n_2
\end{pmatrix}
\end{bmatrix}.
\end{align}
We have a bijection between the quotients
\begin{align}
\Gamma_1(N) \backslash \Delta^{(n)}_1(N) \longrightarrow  \Gamma_1(N) \backslash L_{1,0}^{(n)}, \quad \Gamma_1(N) \begin{bmatrix}
m_1 & n_1 \\ m_2 & n_2
\end{bmatrix} \longmapsto \Gamma_1(N)\begin{bmatrix}
m_1 & n_2 \\ -m_2 & n_1
\end{bmatrix}.
\end{align}
The map is well-defined since the action of $\gamma=\begin{bsmallmatrix}
a & b \\ c & d
\end{bsmallmatrix}$ by $\rho$ on the left becomes multiplication by $\begin{bsmallmatrix}
a & -b \\ -c & d
\end{bsmallmatrix}$ on the right. Each coset representative $\Delta^{(n)}_1(N)\vbf=\Delta^{(n)}_1(N) \begin{bsmallmatrix}
m_1 & n_1 \\ m_2 & n_2
\end{bsmallmatrix}$ sends $\{0,\8\}$ to $T_\vbf \{0,\8\}=\left \{\frac{n_1}{n_2},\frac{m_1}{m_2}\right \}$. From Proposition \ref{properties of Xv}, the submanifold $S_\vbf$ is then \begin{align}
S_{\begin{bsmallmatrix}
m_1 & n_2 \\ -m_2 & n_1
\end{bsmallmatrix}} = \left \{\frac{n_1}{n_2},\frac{m_1}{m_2} \right \}=T_\vbf \{0,\8\}. 
\end{align}
\end{proof}

\subsection{Pullback of the Mathai-Quillen form} 
Let $B \longrightarrow X \cong \HH \times A_\R$ be the rank $2$ vector bundle from \eqref{bundle B defined}. We view it as a metric bundle, where the metric is induced from the Euclidean metric on $\R^2$. It is isomorphic to the trivial bundle $X \times \R^2$, where the metric over $(z,t) \in \HH \times A_\R \cong X$ is given by the quadratic form $t^{-2}z^{-1}$.

 For a vector $\vbf \in V$, let $\Gamma_\vbf \subseteq \Gamma$ be the stabilizer of $\vbf$ in $\Gamma$. By taking the quotient by $\Gamma_\vbf$, we obtain a bundle $\Gamma_\vbf \backslash B \longrightarrow \Gamma_\vbf \backslash X$ over $\Gamma_\vbf \backslash X \cong (\Gamma_\vbf \backslash \HH) \times A_\R$. After fixing an orientation, the integration along the fibers of $B$ induces the Thom isomorphism
\begin{align} \label{thom isomorphism}
H^2_{\rd}(\Gamma_\vbf \backslash B;\R) \longrightarrow H^0(\Gamma_\vbf \backslash X;\R) \cong \R,
\end{align}
where $H^2_{\rd}(\Gamma_\vbf \backslash B;\R)$ is the cohomology of $\Gamma_\vbf$-invariant $2$-forms on $B$ that are rapidly decreasing along the fibers; see \cite[p.~99]{MQ86} or \cite[Theorem.~6.17]{BT82}. The preimage of $1 \in H^0(\Gamma_\vbf \backslash X;\R)$ under the map \eqref{thom isomorphism} is a characteristic class $\Th(\Gamma_\vbf \backslash B) \in H^2_{\rd}(\Gamma_\vbf \backslash B;\R)$ of the bundle, called the {\em Thom class} of $\Gamma_\vbf \backslash B$.

Given a connection that is compatible with the metric of the bundle, Mathai-Quillen \cite{MQ86} construct an explicit differential form
\begin{align}
U \in \Omega_{\rd}^2(B)^{\GL_2(\R)^+}
\end{align}
representing the Thom class. It is a closed, $\GL_2(\R)^+$-invariant (so in particular $\Gamma_\vbf$-invariant), rapidly decreasing form of integral $1$ along the fiber. Let
\begin{align} 
\varphi^0(z,t,\vbf) \coloneqq s_\vbf^\ast U \in \Omega^2(\HH \times A_\R) \otimes C^\8(V_\R)
\end{align}
be the pullback along the section $s_\vbf$. It is closed and satisfies the invariance property
\begin{align}  \label{invariance}
g^\ast \varphi^0(z,t,\vbf)=\varphi^0(z,t,\rho_{g^{-1}}\vbf)
\end{align}
for all $g \in \GL_2(\R)^+$. The form $\varphi^0(z,t,\vbf)$ is smooth in $\vbf$, but not rapidly decreasing. We obtain a rapidly decreasing form by setting
\begin{align} \label{varphi def}
\varphi(z,t,\vbf) \coloneqq e^{-2\pi Q(\vbf)} \varphi^0(z,t,\vbf) \in \Omega^2(\HH \times A_\R) \otimes \Scal(V_\R),
\end{align}
where $\Scal(V_\R)$ is the space of Schwartz functions on $V_\R$. This form is also closed and satisfies
\begin{align} \label{equivariance}
g^\ast \varphi(z,t,\vbf)=\varphi(z,t,\rho_{g^{-1}}\vbf), \qquad g \in \GL_2(\R)^+.
\end{align}
In particular, note that it is also $\Gamma_\vbf$-invariant.

\begin{rmk} Let $\D$ be the symmetric space of $\SO(V)(\R)$, and $\rho \colon X \hooklongrightarrow \D$ the embedding induced from the map \eqref{define rho}. Let $\varphi_{\KM} \in \Omega^2(\D)\otimes \Scal(V_\R)$ be the $2$-form constructed by Kudla-Millson \cite{KM86,KM87,KM90}. It follows from \cite{rbrkmmq} and the functoriality of the Mathai-Quillen form that $\rho^\ast\varphi_{\KM}=\varphi$.
\end{rmk}

\subsection{Eisenstein class} \label{section weil def}  
We consider the Weil representation 
\begin{align} \label{weil representation schrödinger}
\SO(V)(\R) \times \SL_2(\R) \longrightarrow U(\Scal(V_\R))
\end{align}
in the Schrödinger model, where $\SO(V)$ is the orthogonal group of the quadratic space $(V,Q)$ as in Section \ref{tautological bundle}. The map $\rho$ defined in \eqref{define rho} embeds $\GL_2(\R)^+ \cong \SL_2(\R) \times A_\R$ in $\SO(V)(\R)$, and the pullback along $\rho$ yields a representation
\begin{align} \label{weil representation restriction}
\omega \colon \SL_2(\R) \times A_\R \times \SL_2(\R) \longrightarrow  U(\Scal(V_\R)).
\end{align}
Using the explicit formula for the Weil representation \eqref{weil representation schrödinger} that can be found in \cite[Section.~2.5.8, p.~216]{LV80} for example, one find that the restriction \eqref{weil representation restriction} can be explicitly described by the formulas
\begin{alignat}{3}
&\omega(g,t,1) \phi(\vbf)=\phi(\rho_{gt}^{-1}\vbf) & \qquad (g,t) \in \SL_2(\R) \times A_\R, \\
& \omega \left (1, 1,n(x) \right ) \phi(\vbf)=e(xQ(\vbf))\phi(\vbf)   & \qquad  n(x)=\begin{pmatrix}
1 & x \\ 0 & 1
\end{pmatrix},\\
& \omega \left (1, 1,a(y) \right ) \phi(\vbf)=y\phi(\sqrt{y}\vbf)   & \qquad  a(y)=\begin{pmatrix}
\sqrt{y} & 0 \\ 0 & 1/\sqrt{y}
\end{pmatrix},\\
& \omega \left (1, 1, S \right ) \phi(\vbf)= \int_{V_\R} \phi(\vbf') e \left (- Q(\vbf,\vbf') \right )d\vbf' & \qquad S=\begin{pmatrix}
        0 & -1 \\ 1 & 0
    \end{pmatrix},
\end{alignat}
for a Schwartz function $\phi \in \Scal(V_\R)$. For a lattice coset $L$ as in \eqref{linear combination}, we define the theta series
 \begin{align}
\Theta_{L}(z,t,\tau) \coloneqq j(h_\tau,i)^{-2}\sum_{\vbf \in L} \omega(1,1,h_\tau)\varphi(z,t,\vbf) = \sum_{\vbf \in L} \varphi^0(z,t,\sqrt{y}\vbf)e(\tau Q(\vbf)),
\end{align}
where $h_\tau=n(x)a(y)$. By \eqref{equivariance}, it is a $\Gamma$-invariant form in $z$, and transforms like a modular form of weight $2$ in $\tau$ by the theta machinery. Thus it defines an element
\begin{align}
\Theta_L(z,t,\tau) \in \left [\Omega^2(\HH \times A_\R) \otimes C^\8(\HH) \right ]^{\Gamma\times \Gamma},
\end{align}
where $\Gamma$ acts by pullback on $\Omega^2(\HH \times A_\R)$, and by the weight $2$ slash operator on $C^\8(\HH)$.

Let $\pi \colon \HH \times A_\R \longrightarrow \HH$ be the projection onto $\HH$. Given a differential $2$-form
\begin{align}
\omega=f_{uv}(z,t)dudv+f_{ut}(z,t)dudt+f_{vt}(z,t)dvdt \in \Omega^2_{\rd}(\HH \times A_\R)
\end{align}
that is rapidly decreasing as $t \in A_\R$ goes to $0$ and $\8$, its pushforward is defined by
\begin{align}
\pi_\ast \omega= du\left ( \int_{A_\R}f_{ut}(z,t)dt \right )+dv \left ( \int_{A_\R}f_{vt}(z,t)dt \right ) \in \Omega^1(\HH);
\end{align}
see \cite[p.61]{BT82} or \cite[p.~99]{MQ86}. Since $\Gamma$ acts on $X \cong \HH \times A_\R$ by $\gamma (z,t)=(\gamma z,t)$, the action commutes with the pushforward and descends to map
\begin{align}
\pi_\ast \colon \Omega^2_{\rd}(\HH \times A_\R)^\Gamma \longrightarrow \Omega^1(\HH)^\Gamma.
\end{align}
\begin{prop}\label{form is closed} \cite[Proposition.~4.6]{rbrsln} The pushforward
\begin{align}
\Ecal_{L}(z,\tau,s) \coloneqq \pi_\ast \left ( \Theta_{L}(z,t,\tau) t^{2s} \right ) \in \left [\Omega^1(\HH) \otimes C^\8(\HH) \right ]^{\Gamma\times \Gamma}.
\end{align}
converges for any $s \in \C$, and for $\re(s) <-2$ the sum and integral over $A_\R$ can be interchanged. At $s=0$, the form
\begin{align}
\Ecal_{L}(z,\tau) \coloneqq \Ecal_{L}(z,\tau,0)
\end{align}
is a closed differential form (in $z$).
\end{prop}
We will consider the cases $L=L_{\xbf_0}$ and $L=L_{p,q}$, and define
\begin{align}
\Ecal_{\xbf_0}(z,\tau,s) \coloneqq \Ecal_{L_{\xbf_0}}(z,\tau,s), \qquad \Ecal_{p,q}(z,\tau,s) \coloneqq \Ecal_{L_{p,q}}(z,\tau,s).
\end{align}
We denote by $\Ecal_{\xbf_0}(z,\tau)$ and $\Ecal_{p,q}(z,\tau)$ their value at $s=0$.

\subsection{Formulas for $\varphi$} \label{Section : explicit form}Let $H_d(t)\coloneqq \left ( 2t-\frac{d}{dt}\right )^d \cdot 1$ be the $d$-th Hermite polynomial. The first three Hermite polynomials are $H_0(t)=1$, $H_1(t)=2t$ and  $H_2(t)=4t^2-2$.

\begin{prop} \label{explicit form} Let $z=u+iv$ be the coordinates on $\HH$. We have 
\begin{align}
\varphi(z,t,\vbf)=\omega(g_z,t,1) \varphi(\vbf),
\end{align} where $g_z=n(u)a(v) \in \SL_2(\R)$ and
\begin{align}
\varphi(\vbf)= - \left (\varphi_{(2,0)}(\vbf)+ \varphi_{(0,2)}(\vbf)\right )\frac{dudv}{4v^2}+\left ( \left (\varphi_{(0,2)}(\vbf)-\varphi_{(2,0)}(\vbf)\right )\frac{du}{2v} + \varphi_{(1,1)}(\vbf) \frac{dv}{v}\right )\frac{dt}{t},
\end{align}
and the components are
\begin{align}
\varphi_{(2,0)}(\vbf) & \coloneqq \frac{1}{4\pi}\exp \left (-\pi \lVert m \rVert^2-\pi \lVert n \rVert^2 \right )   H_2(\sqrt{\pi} \langle m+n, e_1 \rangle)  \\
\varphi_{(0,2)}(\vbf) &\coloneqq \frac{1}{4\pi}\exp \left (-\pi \lVert m \rVert^2-\pi \lVert n \rVert^2 \right )   H_2(\sqrt{\pi} \langle m+n, e_2 \rangle)   \\
\varphi_{(1,1)}(\vbf) &\coloneqq \frac{1}{4\pi}\exp \left (-\pi \lVert m \rVert^2-\pi \lVert n \rVert^2 \right ) H_1(\sqrt{\pi} \langle m+ n, e_1 \rangle)  H_1(\sqrt{\pi} \langle m+ n, e_2 \rangle).
\end{align}
\end{prop}
\begin{proof}
We briefly recall the computations from \cite{rbrsln} that lead to the explicit expression given in the proposition. 

Let $\vartheta=(gt)^{-1}d(gt)=t^{-1}dt+g^{-1}dg$ be the Maurer-Cartan form on the principal $\SO(2)$-bundle $\SL_2(\R)\times A_\R \longrightarrow \HH \times A_\R$. Let $\lambda \coloneqq \frac{1}{2}(\vartheta+\vartheta^t) \in \Omega^1(\SL_2(\R) \times A_\R) \otimes \Mat(\R^{2})$ and let $\lambda_{ij} \in \Omega^1(\SL_2(\R) \times A_\R)$ be its $(i,j)$-entry. In the coordinates $z=u+iv$ on $\HH$, we have
\begin{align}
\lambda = \frac{1}{2} \left (\vartheta+\vartheta^t \right )= \frac{1}{2} \left (g^{-1}dg+(g^{-1}dg)^t \right )+\frac{dt}{t}= \begin{pmatrix}
\frac{dt}{t} + \frac{dv}{2v}& \frac{du}{2v} \\ \frac{du}{2v} &\frac{dt}{t} - \frac{dv}{2v}
\end{pmatrix}.
\end{align}
For a function $\sigma \colon \{ 1, 2 \} \longrightarrow \{ 1, 2 \}$ we define the $2$-form
\begin{align}
    \lambda(\sigma) \coloneqq \lambda_{1\sigma(1)} \wedge \lambda_{2\sigma(2)} \in \Omega^1(\SL_2(\R) \times A_\R)
\end{align}
and the generalized Hermite polynomial $H_\sigma \in \C[\R^{2}]$ by
\begin{align}
    H_\sigma(a_1,a_2)  \coloneqq H_{d_1}(a_1)H_{d_2}(a_2)
\end{align}
where $d_k = \vert \sigma^{-1}(k) \vert$. With these notations, it was computed in \cite[Proposition.~3.4]{rbrsln}
    \begin{align} \label{formula varphi}
    \varphi(z,t,\vbf)=\frac{1}{4\pi}\sum_{\sigma}  H_\sigma\left (\sqrt{\pi}(t^{-1}g^{-1}m+tg^Tn) \right ) \exp\left (-\pi  \lVert t^{-1}g^{-1}m \rVert^2 -  \lVert tg^Tn \rVert^2 \right ) \lambda(\sigma),
\end{align}
where the sum is over all functions $\sigma \colon \{ 1, 2\} \longrightarrow \{ 1, 2\}$. We have
\begin{align}
\lambda_{11}\lambda_{21} &=-\frac{dudv}{4v^2}-\frac{dudt}{2vt}, \qquad \lambda_{12}\lambda_{22} =-\frac{dudv}{4v^2}+\frac{dudt}{2vt}, \\
\lambda_{12}\lambda_{21}&=0, \qquad \lambda_{11}\lambda_{22}= \frac{dvdt}{vt}.
\end{align}
\end{proof}

For each Schwartz function $\varphi_{(d_1,d_2)}$, let $\Theta_L(z,t,\tau)_{(d_1,d_2)}$ be the corresponding theta series
\begin{align} \label{thetad1d2}
\Theta_L(z,t,\tau)_{(d_1,d_2)}\coloneqq j(h_\tau,i)^{-2} \sum_{\vbf \in L} \omega(g_z,t,h_\tau)\varphi_{(d_1,d_2)}(\vbf).
\end{align}
We have
\begin{align} \label{split ephi}
\Ecal_L(z,\tau,s)= \left (\Ecal_{L}(z,\tau,s)_{(0,2)}-\Ecal_L(z,\tau,s)_{(2,0)} \right )\frac{du}{2v}+\Ecal_L(z,\tau,s)_{(1,1)}\frac{dv}{v}
\end{align}
where the $(d_1,d_2)$-component of $\Ecal_L$ is 
\begin{align}
\Ecal_L(z,\tau,s)_{(d_1,d_2)} & \coloneqq \int_0^\8 \Theta_L(z,t,\tau)_{(d_1,d_2)}t^{2s}\frac{dt}{t}.
\end{align}

\subsection{Functional equation} Let $L'$ be the image of $L$ under the involution
\begin{align} \label{involution}
\begin{bmatrix} m_1 & n_1 \\ m_2 & n_2
\end{bmatrix} \longmapsto \begin{bmatrix} -n_2 & -m_1 \\ n_1 & m_2
\end{bmatrix}.
\end{align}
\begin{prop} \label{functional equation prop} The form $\Ecal_L(z,\tau,s)$ satisfies the functional equation
\begin{align}
\Ecal_L(z,\tau,-s)=-\Ecal_{L'}(z,\tau,s).
\end{align}
In particular, we have $\Ecal_L(z,\tau)=0$ if $L=L'$ (for example if $L=V_\Z$).
\end{prop}
\begin{proof}
From $g^{-T}=S^{-1}gS$ with $S=\begin{psmallmatrix}
0 & 1 \\ -1 & 0 
\end{psmallmatrix}$, we get
\begin{align}
\langle t^{-1}g^{-1}m+tg^Tn, e_k \rangle= \langle t^{-1}g^TSm+tg^{-1}Sn, Se_k \rangle.
\end{align}
Since $Se_1=-e_2$ and $Se_2=e_1$, we see from the formula in Proposition \ref{explicit form} that
\begin{align}
\omega(g,t,h_\tau)\varphi_{(2,0)}(m,n)=\omega(g,t^{-1},h_\tau)\varphi_{(0,2)}(Sn,Sm), \\
\omega(g,t,h_\tau)\varphi_{(1,1)}(m,n)=-\omega(g,t^{-1},h_\tau)\varphi_{(1,1)}(Sn,Sm).
\end{align}
The map $(m,n) \mapsto (Sn,Sm)$ is exactly the involution \eqref{involution}, from which it follows that
\begin{align}
\Theta_L(z,t,\tau)=-\Theta_{L'}(z,t^{-1},\tau).
\end{align}
Combining this with the change of variables $t\mapsto t^{-1}$ in the integral from $0$ to $\8$, we deduce that
\begin{align}
\Ecal_L(z,\tau,s) & =\pi_{\ast} \left ( \Theta_{L}(z,t,\tau) t^{2s} \right ) \\
& = -\pi_{\ast} \left ( \Theta_{L'}(z,t^{-1},\tau) t^{2s} \right ) \\
& = -\pi_{\ast} \left ( \Theta_{L'}(z,t,\tau) t^{-2s} \right ) \\
& =-\Ecal_{L'}(z,\tau,-s).
\end{align}
\end{proof}

\subsection{Fourier expansion of the cohomology class}

Under Poincaré-Lefschetz duality 
\begin{align}
H_1(\YBS,\partial \YBS;\Z) \cong H^1(Y;\Z),
\end{align} the cycle $\Zcal_n(L)$ has a Poincaré dual
\begin{align}
\PD(\Zcal_n(L)) \in H^1(Y;\Z),
\end{align}
which is characterized by the property that for any closed form $\eta \in \Omega_c^1(Y)$, we have
\begin{align}
\int_Y \eta \wedge \PD(\Zcal_n(L)) = \int_{\Zcal_n(L)} \eta.
\end{align}
There is also an intersection pairing
\begin{align} \label{intersection pairing}
\langle - \ , \ - \rangle \colon H_1(Y;\Z) \times H_1(\YBS,\partial \YBS;\Z) \longrightarrow \Z,
\end{align} 
which counts the signed intersection number between the two $1$-cycles. If $\eta_\Zcal$ represents the Poincaré dual in $H^1_c(Y;\Z) \cong H_1(Y;\Z)$ of a cycle $\Zcal \in H_1(Y;\Z)$, then
\begin{align}
\int_Y \eta_\Zcal \wedge \PD(\Zcal_n(L)) = \langle \Zcal, \Zcal_n(L) \rangle
\end{align}
is the intersection number in $Y$ between the two cycles. Since $\Ecal_{L}(z,\tau)$ is closed by Proposition \ref{form is closed}, it defines a cohomology class $[\Ecal_{L}]=[\Ecal_{L}(z,\tau)]$, which can be seen as a nonholomorphic modular of weight $2$, valued in $H^1(Y;\C)$. As in the work of Kudla-Millson, it turns out that this modular form becomes {\em holomorphic} in cohomology, and that the positive Fourier coefficients are Poincaré duals of the special cycles $\Zcal_n(L)$. Since we took $L$ to be any linear combination of cosets $L_{\xbf_0}$, we first consider the latter case in the Theorem \ref{thm fourier} below.

\begin{lem} \label{cohomologous} The forms $\re(E^{(2)}_{p,q}(z)dz)$ and $E^{(2)}_{p,q}(z)dz$ are cohomologous.
\end{lem}

\begin{proof} We show that  $\im (E^{(2)}_{p,q}(z)dz)$ is exact, using the same idea as in \cite[Lemma.~9.6]{bcgcrm}. Consider the weight $0$ Kronecker-Eisenstein series
\begin{align}
   \Kcal_0 \left (s,z,0,\frac{pz+q}{N} \right ) = \frac{\Gamma(s)}{\pi^s}\sideset{}{^\prime}\sum_{m,n \in\Z} \frac{v^s}{\vert mz+n \vert^{2s}} e\left ( \frac{mq-np}{N}\right ),
\end{align}
where $(p,q) \not \equiv (0,0)$ modulo $N$. From
\begin{align}
\frac{\partial}{\partial z}\frac{v^s}{\vert mz+n \vert^{2s}}=-is\frac{v^{s-1}}{(mz+n)^2\vert mz+n \vert^{2s-2}}
\end{align}
we deduce that 
\begin{align}
\frac{\partial}{\partial z} \Kcal_0 \left (s,z,0,\frac{pz+q}{N} \right )= 4is\frac{\pi^{s+1}} {\Gamma(1+s)}\Kcal_2\left (1+s,z,0,\frac{pz+q}{N} \right )
\end{align}
for $\re(s) \gg 0$. Since $(p,q) \not \equiv (0,0)$, the functions admit an analytic continuation to the entire plane and at $s=1$ we have
\begin{align}
\frac{\partial}{\partial z}\Kcal_0 \left (1,z,0,\frac{pz+q}{N} \right ) =4i \pi \Kcal_2\left (2,z,0,\frac{pz+q}{N} \right )=4i \pi E^{(2)}_{p,q}(z).
\end{align}
Similarly, we have
\begin{align}
\frac{\partial}{\partial \bar{z}} \Kcal_0 \left (1,z,0,\frac{pz+q}{N} \right ) =-4 i \pi \overline{E^{(2)}_{p,q}(z)},
\end{align}
from which we conclude that
\begin{align}
\frac{1}{4i\pi} d \Kcal_0 \left (1,z,0,\frac{pz+q}{N} \right ) = E^{(2)}_{p,q}(z)dz - \overline{E^{(2)}_{p,q}(z)dz}=2i\im\left ( E^{(2)}_{p,q}(z)dz \right ).
\end{align}
\end{proof}

\begin{thm} \label{thm fourier} Let $\xbf_0 = \begin{psmallmatrix} p & k \\q & l\end{psmallmatrix} \in \Mat_2(\Z/N\Z)$ be a nonzero matrix. The cohomology class of $[\Ecal_{\xbf_0}]$ has the Fourier expansion
\begin{align}
[\Ecal_{\xbf_0}] =[a(\xbf_0)]-\sum_{n=1}^\8 \PD(\Zcal_n(\xbf_0)) e(n\tau) \in H^1(Y(N);\C) \otimes M_2(\Gamma(N)),
\end{align}
 where $[a(\xbf_0)] \in H^1(Y(N);\C)$ is the cohomology class represented by the $1$-form
 \begin{align}
a(\xbf_0)=\delta_{k0}\delta_{l0}E^{(2)}_{-q,p}(z)dz-\delta_{p0}\delta_{q0}E^{(2)}_{k,l}(z)dz.
\end{align}
\end{thm}

\begin{proof} This is a special case of the results in \cite{rbrsln}, except the computation of the constant term.  We briefly recall the proof. Since the theta series is absolutely convergent, we can group the vectors of length $n$ and write
\begin{align}
\Theta_{L_{\xbf_0}}(z,t,\tau) = \Theta_{L_{\xbf_0}}^{(0)}(z,t,y)+ \sum_{\substack{n \in \Z \\ n \neq 0}}\Theta^{(n)}_{L_{\xbf_0}}(z,t,y) e(n\tau),
\end{align}
where the $n$-th Fourier coefficient is
\begin{align}
\Theta^{(n)}_{L_{\xbf_0}}(z,t,y) \coloneqq \sum_{\substack{[\vbf] \in \Gamma(N) \backslash L_{\xbf_0} \\ Q(\vbf)=n}} \sum_{\wbf \in \Gamma \vbf} \varphi^0(z,t,\sqrt{y}\wbf).
\end{align}
For all regular vectors $\vbf$, the term
\begin{align}
\sum_{\wbf \in \Gamma \vbf} \varphi^0(z,t,\sqrt{y}\wbf) \in \Omega^2(Y(N) \times A_\R)\otimes C^\8(\R_{>0})
\end{align}
is a Poincaré dual to the cycle $X_{[\vbf]} \in H_2^{\BM}(Y(N) \times A_\R;\Z)$, represented by the image of $X_\vbf$ modulo $\Gamma$. Taking the pushforward of each Fourier coefficient $\Theta^{(n)}_{L_{\xbf_0}}(z,t,y)$ gives the Fourier expansion
 \begin{align}
\Ecal_{\xbf_0}(z,\tau) = \Ecal^{(0)}_{\xbf_0}(z,y)+ \sum_{\substack{n \in \Z \\ n \neq 0}} \Ecal^{(n)}_{\xbf_0}(z,y) e(n\tau),
\end{align}
where
\begin{align}
\Ecal^{(n)}_{\xbf_0}(z,y) \coloneqq \pi_\ast \left ( \Theta^{(n)}_{L_{\xbf_0}}(z,t,y) \right )\in \Omega^1(Y(N)) \otimes C^\8(\R_{>0}).
\end{align}
Interchanging the sum over the regular vectors and the integral over $A_\R$ is valid for any $s \in \C$. For the sum over the singular vectors, the exchange can only be done for $\re(s)$ large enough.

 For regular vectors $\vbf$, the pushforward
\begin{align}
  \pi_\ast \left ( \sum_{\wbf \in \Gamma(N) \vbf} \varphi^0(z,t\sqrt{y}\wbf) \right )   \in \Omega^1(Y(N))\otimes C^\8(\R_{>0})
\end{align}
is a Poincaré dual to $\Zcal_{[\vbf]}$. In particular, this form is exact if $\vbf$ is a nonpositive regular vector (since $X_\vbf$ is empty by Proposition \ref{properties of Xv}). For positive $n$, we sum over positive regular vectors  and 
\begin{align}
\Ecal^{(n)}_{\xbf_0}(z,y)= \sum_{\substack{[\vbf] \in \Gamma(N) \backslash L_{\xbf_0} \\ Q(\vbf)=n}}  \pi_\ast \left ( \sum_{\wbf \in \Gamma(N) \vbf} \varphi^0(z,t\sqrt{y}\wbf) \right )
\end{align} is a Poincaré dual to $\Zcal_n(L_{\xbf_0})$, by the definition given in \eqref{definition of SL}.

Finally, the remaining term is the sum over the singular vectors
\begin{align} \label{singular sum}
\Ecal_{\xbf_0}^{(0)}(z,y)=\pi_\ast \left ( \sum_{\substack{\vbf \in L_{\xbf_0} \\ \vbf \ \textrm{singular}}} \varphi^0(z,t,\sqrt{y}\vbf)  \right )=\pi_\ast \left ( \sum_{\substack{\vbf \in L_{\xbf_0} \\ \vbf \ \textrm{singular}}} \varphi(z,t,\sqrt{y}\vbf)  \right ),
\end{align}
where in the second equality we use that $\varphi(z,t,\vbf) \coloneqq e^{-2\pi Q(\vbf)} \varphi^0(z,t,\vbf)$ (in particular they are equal if $Q(\vbf)=0$). This sum splits into the sum over the singular vectors of the form $\vbf=\begin{bsmallmatrix} m_1 & 0 \\ m_2 & 0\end{bsmallmatrix}$, and those of the form $\vbf = \begin{bsmallmatrix} 0 & n_1 \\ 0 & n_2\end{bsmallmatrix}$ in $L_{\xbf_0}$.

First, suppose that $\begin{bsmallmatrix} k \\ l \end{bsmallmatrix} \equiv \begin{bsmallmatrix} 0 \\ 0 \end{bsmallmatrix}\mod{N}$, so that $\xbf_0 \equiv \begin{bsmallmatrix} p & 0 \\ q & 0 \end{bsmallmatrix} $. In that case, the singular vectors are of the form $\vbf=\begin{bsmallmatrix} m_1 & 0 \\ m_2 & 0\end{bsmallmatrix}$, with $\begin{bsmallmatrix} m_1 \\ m_2 \end{bsmallmatrix} \in \begin{bsmallmatrix} p+N\Z \\ q+N\Z \end{bsmallmatrix}$. Using the explicit forms given in Proposition \ref{explicit form}, we find that

\begin{align} \sum_{\substack{\vbf \in L_{\xbf_0} \\ \vbf \ \textrm{singular}}} \iota_{t \frac{\partial}{\partial t}}\varphi(z,t,\sqrt{y}\vbf) 
 & = \sum_{\substack{m_1 \in p+N\Z \\ m_2 \in q+N\Z}} ye^{-\pi y \frac{\vert z m_2-m_1 \vert^2}{t^2v}} \left ( \frac{v^2m_2^2-(m_1-um_2)^2}{vt^2} \frac{du}{2v}+\frac{(m_1-um_2)vm_2}{vt^2} \frac{dv}{v}\right ) \\
& = -\sum_{\substack{m_1 \in p+N\Z \\ m_2 \in q+N\Z}} \frac{y}{2v^{2}t^2}e^{-\pi y \frac{\vert z m_2-m_1 \vert^2}{vt^2}} \re \left ( \overline{z m_2-m_1}^2dz \right ),
\end{align}
where $\iota_{t \frac{\partial}{\partial t}}$ is the contraction with the vector field $t \frac{\partial}{\partial t}$ on $X \times A_\R$. After adding a term $t^{-2s}$, a direction computation\footnote{Using $\int_0^\8 e^{-at}t^s\frac{dt}{t}=a^{-s}\Gamma(s)$ for $a>0$ and $\re(s)>0$.} shows that
\begin{align}
\pi_\ast \left ( \sum_{\substack{\vbf \in L_{\xbf_0} \\ \vbf \ \textrm{singular}}} \varphi(z,t,\sqrt{y}\vbf)t^{-2s}  \right ) & = \int_{A_\R} \left (\sum_{\substack{\vbf \in L_{\xbf_0} \\ \vbf \ \textrm{singular}}} \iota_{t \frac{\partial}{\partial t}}\varphi(z,t,\sqrt{y}\vbf) \right ) t^{-2s}\frac{dt}{t} \\
& =  -\sum_{\substack{m_1 \in p+N\Z \\ m_2 \in q+N\Z}} \int_{A_\R} \frac{y}{2v^2}e^{-\pi y\frac{\vert z m_2-m_1 \vert^2}{vt^2}} \re \left ( \overline{z m_2-m_1}^2dz \right ) t^{-2s-2}  \frac{dt}{t} \\
& =-\frac{\pi^{-s}y^{-s}}{4\pi v^{1-s}} \Gamma(s+1) \re \left ( \sideset{}{^\prime}\sum_{\substack{m,n \in\Z \\ m_2 \equiv q \mod{N} \\ m_1 \equiv p \mod{N}}} \frac{\overline{z m_2-m_1}}{(z m_2-m_1) \vert z m_2-m_1 \vert^{2s}} \right ),
\end{align}
where the right hand-side converges for $\re(s) \gg 0$. At $s=0$, this gives us
\begin{align}
\Ecal_{\xbf_0}^{(0)}(z,y)=\re \left ( \Ehat^{(1,1)}_{q,-p}(z) dz \right ),
\end{align}
where we recall that
\begin{align}
\Ehat^{(1,1)}_{a,b}(z) & = \Kcal_{2} \left (1,z,\frac{az+b}{N},0 \right ) = -\frac{1}{4 \pi v} \  \lim_{s \rightarrow 0} \sideset{}{^\prime}\sum_{\substack{m,n \in\Z \\ m \equiv a \mod{N} \\ n \equiv b \mod{N}}} \frac{\overline{mz+n}}{(mz+n) \vert mz+n \vert^{2s}}.
\end{align}
By the functional equation \eqref{functional equation Gkl}, we have
\begin{align}
\re \left ( \Ehat^{(1,1)}_{q,-p}(z) dz \right )=\re \left (E^{(2)}_{q,-p}(z)dz\right ),
\end{align}
which is cohomologous to $E^{(2)}_{q,-p}(z)dz$ by Lemma \ref{cohomologous}. Finally, noting that $E^{(2)}_{q,-p}(z)=E^{(2)}_{-q,p}(z)$, we conclude that
\begin{align}
[\Ecal_{\xbf_0}^{(0)}(z,y)]= [E^{(2)}_{-q,p}(z)dz] \qquad  \qquad \textrm{if} \ \ \  \xbf_0 \equiv \begin{bsmallmatrix} p & 0 \\ q & 0 \end{bsmallmatrix},
\end{align}
in cohomology. Note that the right-hand side does not depend on $y$ anymore, since the integral from $0$ to $\8$ is invariant under rescaling by $\sqrt{y}$.

Next, suppose that $\begin{bsmallmatrix} p \\ q \end{bsmallmatrix}  \equiv \begin{bsmallmatrix} 0 \\ 0 \end{bsmallmatrix} \mod N$, so that $\xbf_0=\begin{bsmallmatrix} 0 & k \\0 & l\end{bsmallmatrix}$. In that case, we only have singular vectors of the form $\vbf = \begin{bsmallmatrix} 0 & n_1 \\ 0 & n_2\end{bsmallmatrix}$. From Proposition \ref{functional equation prop}, we have the functional equation $\Ecal_{\xbf_0}(z,\tau,s)=-\Ecal_{\xbf_0'}(z,\tau,-s)$ where $\xbf_0'=\begin{bsmallmatrix} -l & 0 \\ k & 0\end{bsmallmatrix}$. Using the previous computation, we find that
\begin{align}
[\Ecal_{\xbf_0}^{(0)}(z,y)]= -[E^{(2)}_{k,l}(z)dz] \qquad  \qquad \textrm{if} \ \ \  \xbf_0 \equiv \begin{bsmallmatrix} 0 & k \\ 0 & l \end{bsmallmatrix}.
\end{align}
Finally, in the remaining case where $\begin{bsmallmatrix} p \\ q \end{bsmallmatrix}$ and $\begin{bsmallmatrix} k \\ l \end{bsmallmatrix}$ are both non-zero modulo $N$, there are no singular vectors and
\begin{align}
\Ecal_{\xbf_0}^{(0)}(z,y)=0.
\end{align}
\end{proof}

Let us now consider the lattice cosets $L=L_{p,q}$ with $(p,q) \not \equiv (0,0)$ modulo $N$, where we recall that
\begin{align}
L_{p,q}=\sum_{k,l \in \Z/N\Z } L_{\begin{bsmallmatrix} p & k \\ q & l \end{bsmallmatrix}}.
\end{align}
It is preserved by $\Gamma_1(N)$ if $(p,q) \equiv (1,0)$. From the invariance of the form $\varphi^0$ in \eqref{invariance}, it follows that for $\gamma=\begin{psmallmatrix} a & b \\ c & d\end{psmallmatrix} \in \SL_2(\Z)$ we have
\begin{align} \label{invariance of Epq}
\gamma^\ast \Ecal_{p,q}(z,\tau) = \Ecal_{(p,q)\gamma^{-t}}(z,\tau)=\Ecal_{dp-bq,aq-pc}(z,\tau).
\end{align}

For $(a,b) \not \equiv (0,0)$ modulo $N$, we recall the classical Siegel function from \cite[Section.~1.9]{K04}
\begin{align}
g_{a,b}(\tau) \coloneqq q_\tau^{w} \prod_{n=1}^\8 (1-q_\tau^{n+\frac{a}{N}}\zeta_N^b)(1-q_\tau^{n-\frac{a}{N}}\zeta_N^{-b}),
\end{align}
where $w=\frac{1}{12}-\frac{a}{2N}+\frac{1}{2}\frac{a}{N^2}$ and $q_\tau=e^{2i \pi \tau}$. It is known that the functions $g_{a,b}(\tau)^{12N}$ are modular functions for $\Gamma(N)$, so that $g_{a,b} \in \Ocal(Y(N))^\times \otimes \Q$; see \cite{KL81}. If $(a,b) \equiv (0,1)$, then $g_{0,1} \in \Ocal(Y_1(N))^\times \otimes \Q$.
\begin{lem} \label{lemBZ23} For $(p,q) \not \equiv (0,0)$ we have $2i\pi E^{(2)}_{p,q}(z)dz=-d \log ( g_{p,q}(z))$.
\end{lem}
\begin{proof} See \cite[Lemma.~37]{BZ23}. The form $E^{(2)}_{p,q}(z)$ is the form $-E_{\mathbf{x}}^{(2)}(z)$ in {\em loc. cit.} with $\mathbf{x}=(p/N,q/N)$.
\end{proof}
\begin{lem} \label{lemma constant} Let $f \in M_2(\Gamma_1(N))$ be a modular form such that the Fourier coefficients $a_n(f)$ are rational for $n>0$. Then the constant term $a_0(f)$ is a rational number.
\end{lem}
\begin{proof} See \cite[Corollary.~12.3.10]{DI95} or \cite[Proposition.~1.3]{S78}.
\end{proof}
For $(p,q)\equiv (1,0)$, we recover Theorem \ref{thm 1 intro} from the introduction:
\begin{cor} \label{Fourier expansion E10} The class $[\Ecal_{1,0}]$ has the Fourier expansion
\begin{align}
[\Ecal_{1,0}] = -\frac{1}{2i \pi} [d \log(g_{0,1})]-\sum_{n=1}^\8 \PD(T_n\{0,\8\}) e(n\tau) \in H^1(Y_1(N);\Q) \otimes M_2(\Gamma_1(N)).
\end{align}
\end{cor}
\begin{proof}  The Fourier expansion follows from Theorem \ref{thm fourier}, Proposition \ref{hecke translate} and Lemma \ref{lemBZ23}. Since the Fourier coefficients for $n>0$ are Poincaré duals of integral cycles, they define rational classes in $H^1(Y_1(N);\Q)$ (in fact even integral classes). For any cycle $\Zcal \in H_1(Y_1(N);\Q)$, we have
\begin{align}
\Ecal(\Zcal)=- \frac{1}{2i \pi}\int_{\Zcal}  d \log(g_{0,1})  -\sum_{n =1}^\8 \langle \Zcal,T_n\{0,\8\} \rangle e(n\tau) \in M_2(\Gamma_1(N))
\end{align}
where $\langle \Zcal,T_n\{0,\8\} \rangle \in \Q$ is the intersection number between the two cycles. By Lemma \ref{lemma constant}, the constant term of $\Ecal(\Zcal)$ is also rational. Hence, the class $(2i\pi)^{-1}[d \log(g_{0,1})]$ has rational periods and the constant term
\begin{align}
[E^{(2)}_{0,1}(z)dz]=-\frac{1}{2i \pi} [d \log(g_{0,1})] \in H^1(Y_1(N);\Q)
\end{align}
is also rational.
\end{proof}

\subsection{Spectral expansion} Let us now pass to complex coefficients, and view $[\Ecal_{1,0}]$ as an element in
\begin{align}
[\Ecal_{1,0}] \in H^1(Y_1(N);\C) \otimes M_2(\Gamma_1(N)).
\end{align}

Let $\kappa\coloneqq \vert C_N \vert - 1= \dim_\C (E_2(\Gamma_1(N))$ be the dimension of the space of Eisenstein series, and
\begin{align}
E_2(\Gamma_1(N))=\Span_\C \{ E_1(\tau), \dots, E_{\kappa}(\tau)\}
\end{align}
a basis of normalized eigenforms for all Hecke operators $\{T_n \ , \ (n,N)=1\}$; see the basis given in \cite[Theorem.~4.6.2]{DS05} and \cite[Theorem.~5.2.3]{DS05}.  
Let $S_2^{\old}(\Gamma_1(N)) \subset S_2(\Gamma_1(N))$ be the subspace of old forms. Its orthogonal complement is denoted by $S^\new_2(\Gamma_1(N))$, and let $\Bcal^\new$ be the set of newforms {\em i.e.} the set of normalized Hecke eigenforms in $S^\new_2(\Gamma_1(N))$. By \cite[Theorem.~5.8.2]{DS05}, it forms an orthogonal basis of $S^\new_2(\Gamma_1(N))$ and we have
\begin{align}
T_nf=a_n(f)f, \qquad \textrm{for all} \ n>0, \ \ f \in \Bcal^\new.
\end{align}
This basis can be completed to an orthogonal basis $\Bcal$ of $S_2(\Gamma_1(N))$ of normalized eigenforms for all $T_n$ with $(n,N)=1$; see \cite[Theorem.~5.8.3]{DS05} for more details.

By Eichler-Shimura, we have an isomorphism
\begin{align}
M_2(\Gamma_1(N)) \oplus  S_2(\Gamma_1(N)) \longrightarrow H^1(Y_1(N); \C), \quad (f,\overline{g}) \longmapsto \omega_f+\overline{\omega_{g}},
\end{align}
where $\omega_f=f(z)dz$ and $\overline{\omega_{g}}= \overline{g(z)dz}$; see \cite[Theorem.~6.15]{W19} or \cite[Eq.12.2.3]{DI95}. Thus, the cohomology $H^1(Y_1(N);\C)$ is spanned by the classes of $1$-forms $\omega_{E_j}$, $\omega_f$ and $\overline{\omega_f}$ with $1 \leq j \leq \kappa$ and $f \in \Bcal$.
The perfect pairing
\begin{align}
H_1(Y_1(N);\C) \times H^1(Y_1(N);\C) \longrightarrow \C, \qquad [ \Zcal,\omega ] \coloneqq \int_{\Zcal}\omega
\end{align}
induces an isomorphism $H_1(Y_1(N);\C) \simeq H^1(Y_1(N);\C)^\vee$. Hence, for each of the classes $\omega_{E_j}$, $\omega_f$, and $\overline{\omega_f}$, there is a a dual class in $H_1(Y_1(N);\C)$ that we will denote $\Zcal_{E_j}, Z_f$ and $ \overline{Z_f}$ respectively.
The pairing of a form $\omega_f$ against a modular cap $\Ccal_r$ around the cusp $r$ computes the constant term of $f$ at $r$. Hence, we have $[\Ccal_r,\omega_f]=[\Ccal_r,\overline{\omega_f}]=0$, and so
\begin{align}
\Zcal_{E_j} \in \Ccal(\C).
\end{align}
Since $\dim_\C E_2(\Gamma_1(N))=\dim_\C \Ccal(\C)=\kappa$, it follows that 
\begin{align}
\Span_\C \{ \Zcal_{E_j} \ \vert \ 1 \leq j \leq \kappa\}=\Ccal(\C).
\end{align}
Thus, we have a splitting
\begin{align} \label{splitting cusp eis homology}
H_1(Y_1(N);\C)=\Ccal(\C) \oplus \Span_\C\{ \Zcal_f, \overline{\Zcal_f} \ \vert \ f \in \Bcal \}.
\end{align}
\begin{rmk}One could explicitly write $\Zcal_{E_j}$ as a linear combinations of the modular caps $\Ccal_r$ by computing the constant term of each $E_j$ at each cusp $r$.
\end{rmk}
By Poincaré-Lefschetz duality, we also have an isomorphism
\begin{align}
H^1(Y_1(N);\C) \cong H_1(\YBSN,\partial \YBSN;\C).
\end{align}
If $\omega_{\Zcal'} \in H^1(Y_1(N);\C)$ is the Poincaré dual of a relative cycle $\Zcal' \in H_1(\YBSN,\partial \YBSN;\C)$, then 
\begin{align}
[ \Zcal,\omega_{\Zcal'}]=\langle \Zcal,\Zcal' \rangle,
\end{align}
where $\langle \Zcal,\Zcal' \rangle$ is the intersection number on $Y_1(N)$ as in \eqref{intersection pairing}.

\begin{lem} \label{pairing lemme} We have
\begin{align}
[\Zcal_f, \PD(\{0,\8\})] =-\frac{\overline{L(f,1)}}{4 \pi \lVert f \rVert^2}, \qquad
[\overline{\Zcal_f}, \PD(\{0,\8\})] =\frac{L(f,1)}{4 \pi \lVert f \rVert^2},
\end{align}
where $\lVert f \rVert^2$ is the Petersson norm.
\end{lem}
\begin{proof} By Poincaré-Lefschetz duality, we also have an isomorphism
\begin{align}
H_1(Y_1(N);\C) \cong H^1(Y_1(N);\C)^\vee \cong H^1(\YBSN,\partial \YBSN;\C),
\end{align}
where the second isomorphism is induced from the pairing
\begin{align}
H^1(\YBSN,\partial \YBSN;\C) \times H^1(Y_1(N);\C) \longrightarrow \C, \qquad (\eta,\omega)=\int_{Y_1(N)} \eta \wedge \omega,
\end{align}
the class $\omega$ is represented by one of the forms $\omega_{E_j},\omega_f$ or $\overline{\omega_f}$, and $H^1(\YBSN,\partial \YBSN;\C)$ is the cohomology of forms vanishing at the boundary $\partial \YBSN$.
For $f \in S_2(\Gamma_1(N))$, the forms $\omega_f$ and $\overline{\omega_f}$ also define classes in $H^1(\YBSN,\partial \YBSN;\C)$. From
\begin{align}
(\overline{\omega_f},\omega_f)&=\int_{Y_1(N)} \overline{\omega_f} \wedge \omega_f = 2i \lVert f \rVert^2, \\
(\overline{\omega_f},\overline{\omega_f})&=\int_{Y_1(N)} \overline{\omega_f} \wedge \overline{\omega_f}= 0, \\
(\overline{\omega_f},\omega_{E_j})&=\int_{Y_1(N)} \overline{\omega_f} \wedge \omega_{E_j}=0,
\end{align}
it follows that dual of $\omega_f \in H^1(Y_1(N);\C)$ is
\begin{align}
\omega_f^\vee=\frac{1}{2i \lVert f \rVert^2} \overline{\omega_f} \in H^1(\YBSN,\partial \YBSN;\C).
\end{align}
The $1$-cycle $\Zcal_f$ is characterized by the property that
\begin{align}
[\Zcal_f, \omega]=(\omega_f^\vee,\omega)
\end{align}
for all classes $\omega \in H^1(Y_1(N);\C)$, and thus we deduce that
\begin{align}
[\Zcal_f, \PD(\{0,\8\})]& =(\omega_f^\vee,\PD(\{0,\8\}) \\
& = \frac{1}{2i \lVert f \rVert^2} \int_{Y_1(N)} \overline{\omega_f} \wedge \PD(\{0,\8\}) \\
& =\frac{1}{2i \lVert f \rVert^2}\int_0^\8 \overline{\omega_f}.
\end{align}
The first equality then follows from the fact that
\begin{align}
\int_0^\8 \overline{\omega_f}=\frac{1}{2i\pi}\overline{L(f,1)}.
\end{align}
Analogously, we find that the dual of $\overline{\omega_f}$ is
\begin{align}
\overline{\omega_f}^\vee=-\frac{1}{2i \lVert f \rVert^2} \omega_f \in H^1(\YBSN,\partial \YBSN;\C).
\end{align}
The second equality follows from similar computations, using that
\begin{align}
\int_0^\8 \omega_f=-\frac{1}{2i\pi}L(f,1).
\end{align}
\end{proof}

Let $F \in  H^1(Y_1(N);\C) \otimes S_2(\Gamma_1(N))$ be the element defined by
\begin{align}
F \coloneqq -\sum_{j=1}^\kappa \langle \Zcal_{E_j},\{0,\8\} \rangle \omega_{E_j} \otimes E_j + \frac{1}{4 \pi} \sum_{f \in \Bcal} \frac{\overline{L(f,1)}}{\lVert f \rVert^2} \omega_{f} \otimes f  -\frac{1}{4 \pi} \sum_{f \in \Bcal} \frac{L(f,1)}{\lVert f \rVert^2} \overline{\omega_{f}} \otimes f^\sigma,
\end{align}
where $f^\sigma \in S_2(\Gamma_1(N))$ is the cusp form obtained from $f$ by applying complex conjugation to the Fourier coefficients.
\begin{thm} \label{Spectral expansion}  We have
\begin{align}
[\Ecal_{1,0}]- F \in H^1(Y_1(N);\C) \otimes M^\old_2(\Gamma_1(N)).
\end{align}
In particular, if $N$ is prime, then $[\Ecal_{1,0}]=F$.
\end{thm}
\begin{proof} We proceed as in \cite[Theorem.~3.5]{DPV}.  We write the class $\PD(\{0,\8\})$ as
\begin{align} \label{equality of PD}
\PD(\{0,\8\})=\sum_{j=1}^\kappa \lambda_j \omega_{E_j}+ \sum_{f \in \Bcal} \alpha_f \omega_f + \sum_{f \in \Bcal} \beta_f  \overline{\omega_f} \in H^1(Y_1(N);\C).
\end{align}

By pairing both sides of \eqref{equality of PD} against $\Zcal_f$ and $\overline{\Zcal_f}$, it follows from Lemma \ref{pairing lemme} that
\begin{align}
\alpha_f=-\frac{1}{4\pi} \frac{\overline{L(f,1)}}{\lVert f \rVert^2},\qquad \beta_f=\frac{1}{4\pi} \frac{L(f,1)}{\lVert f \rVert^2}.
\end{align}
On the other hand, by pairing both sides of \eqref{equality of PD} against the $1$-cycle $\Zcal_{E_j}$, which is dual to $\omega_{E_j}$, we also find that
\begin{align}
\lambda_j=[\Zcal_{E_j},\PD(\{0,\8\})]=\langle \Zcal_{E_j},\{0,\8\} \rangle.
\end{align}
If we apply the Hecke operator $T_n$ on both sides with $(n,N)=1$, then
\begin{align} \label{equality of TnPD}
\PD(T_n\{0,\8\})=\sum_{j=1}^\kappa \lambda_j a_n(E_j)\omega_{E_j}+ \sum_{f \in \Bcal} \alpha_f a_n(f)\omega_{f} + \sum_{f \in \Bcal} \beta_f  \overline{a_n(f)}\overline{\omega_{f}},
\end{align}
where $a_n(f)$ is the $n$-th Fourier coefficient of $f$.
By taking the $q$-expansion on both sides of \eqref{equality of TnPD} and comparing with Corollary \ref{Fourier expansion E10}, it follows that $[\Ecal_{1,0}]$ and $F$ have the same Fourier coefficients for all $(n,N)=1$, which means that their difference is an oldform.  Since there are no oldforms when $N$ is prime, we deduce that $[\Ecal_{1,0}]=F$ in that case.
\end{proof}

\subsection{A theta lift}

The integral of $\Ecal_{1,0}$ over a $1$-cycle defines a lift 
\begin{align} \label{lift in text}
\Ecal \colon H_1(Y_1(N);\C) \longrightarrow M_2(\Gamma_1(N)), \qquad \Ecal(\Zcal) = [\Zcal,\Ecal_{1,0}].
\end{align}
\begin{mydef}
Let $S_{2,\rk=0}^\new(\Gamma_1(N))$ be the subspace spanned by the newforms $f$ with $L(f,1) \neq 0$.
\end{mydef}
\begin{thm} \label{surjective morphism}  If $f$ is a newform, then 
\begin{align}
\Ecal(\Zcal_f)= \frac{1}{4 \pi} \frac{\overline{L(f,1)}}{\lVert f \rVert^2} f , \qquad \Ecal(\overline{\Zcal_f})=-\frac{1}{4 \pi} \frac{L(f,1)}{\lVert f \rVert^2} f^\sigma.
\end{align}
In particular, we have
\begin{align}
S_{2,\rk=0}^\new(\Gamma_1(N)) \subseteq \Ecal \left ( \Span_\C\{ \Zcal_f, \overline{\Zcal_f} \ \vert \ f \in \Bcal \} \right ),
\end{align}
with equality when $N$ is prime.
\end{thm}

\begin{proof} When $N$ is prime, it follows from Theorem \ref{Spectral expansion} that the spectral expansion of $[\Ecal_{1,0}]$ is \begin{align}
[\Ecal_{1,0}]=-\sum_{j=1}^\kappa \langle \Zcal_{E_j},\{0,\8\} \rangle \omega_{E_j} \otimes E_j + \frac{1}{4 \pi} \sum_{f \in \Bcal} \frac{\overline{L(f,1)}}{\lVert f \rVert^2} \omega_{f} \otimes f  -\frac{1}{4 \pi} \sum_{f \in \Bcal} \frac{L(f,1)}{\lVert f \rVert^2} \overline{\omega_{f}} \otimes f^\sigma.
\end{align}
The theorem follows immediately, since $\Bcal=\Bcal^\new$ when $N$ is prime. Note that if $f$ is a newform, then so is $f^\sigma$; see \cite[Corollary.~12.4.5]{DI95}.
Now suppose that $N$ is not prime. Let $f$ be a newform with $L(f,1) \neq 0$. The image of $\Zcal_f$ is the modular form $[\Zcal_f,\Ecal_{1,0}]$, whose $n$-th Fourier coefficient is
$[\Zcal_f, \PD(T_n\{0,\8\})]$. By \eqref{equality of TnPD}, we have
\begin{align}
[\Zcal_f, \PD(T_n\{0,\8\})]=-\frac{1}{4\pi} \frac{L(f,1)}{\lVert f \rVert^2}a_n(f).
\end{align}
for all $(n,N)=1$. Since $f$ is a newform and the Hecke operators preserve the space of oldforms and newforms, this equality actually holds {\em for all} $n \geq 1$. By taking the $q$-expansion on both sides, it follows that
\begin{align}
\Ecal(\Zcal_f)=\frac{1}{4\pi} \frac{\overline{L(f,1)}}{\lVert f \rVert^2}f,
\end{align}
and similarly for $\Ecal(\overline{\Zcal_f})$.
\end{proof}

\subsection{Poisson summation} Theorem \ref{surjective morphism} gives no information about the image of $\Ccal(\C)$ in the decomposition \eqref{splitting cusp eis homology}. To compute the integral of $\Ecal_{1,0}$ along a modular cap $\Ccal_r$, we compute its constant term (in $z$) at the corresponding cusp $r$. Note that this differs from Theorem \ref{thm fourier}, where the Fourier expansion was taken in the variable $\tau$. To permute the variables $z$ and $\tau$, we define a representation
\begin{align}
\omega'(g,t,h) \coloneqq \omega(h,t,g),
\end{align}
where $\omega(g,t,h)$ is the Weil representation from \eqref{section weil def}. These two representations $\omega$ and $\omega'$ are intertwined by an operator $\Fcal \colon \Scal(V_\R) \longrightarrow \Scal(V_\R)$, where $\Fcal(\varphi) \coloneqq \vphat$ is the partial Fourier transform
\begin{align} \label{fourier transform 1}
\vphat\left (\begin{bmatrix}
m_1 & n_1\\ m_2 & n_2
\end{bmatrix} \right ) \coloneqq \int_{\R^2} \varphi \left (\begin{bmatrix}
a & n_1\\ m_2 & b
\end{bmatrix} \right )e(an_2-bm_1) dadb.
\end{align}A direct computation shows that
\begin{align} \label{intertwine fcal}
\Fcal \circ \omega = \omega' \circ \Fcal.
\end{align}
For a Schwartz function $f \in \Scal(\R^2)$ and an integer $p \in \Z$, we have a shifted Poisson summation formula 
\begin{align}
\sum_{\substack{\mu \in p+N\Z \\ \nu \in \Z }}f(\mu,\nu)= \frac{1}{N} \sum_{\substack{\mu \in \frac{1}{N}\Z \\ \nu \in \Z}}e(p\mu)\int_{\R^2} f(a,b)e(-a\mu -b\nu)dadb.
\end{align}
After changing $(\mu,\nu) \mapsto (-\nu,\mu)$ in the right-hand side, we can also rewrite it as
\begin{align}
\sum_{\substack{\mu \in p+N\Z \\ \nu \in \Z }}f(\mu,\nu)= \frac{1}{N} \sum_{\substack{\mu \in \Z \\ \nu \in \frac{1}{N}\Z}}e(-p\nu)\int_{\R^2} f(a,b)e(a\nu -b\mu)dadb.
\end{align}
Applying this formula (in the variables $(\mu,\nu)=(m_1,n_2)$) to the theta series \eqref{thetad1d2}, together with \eqref{intertwine fcal}, yields
\begin{align}
\Theta_{p,q}(z,t,\tau)_{(d_1,d_2)}& =\frac{1}{y}\sum_{\begin{bsmallmatrix} m_1 & n_1 \\ m_2 & n_2\end{bsmallmatrix} \in L_{p,q}} \omega(g_z,t,h_\tau)\varphi_{(d_1,d_2)}\left (\begin{bsmallmatrix} m_1 & n_1 \\ m_2 & n_2\end{bsmallmatrix}\right ) \\
&=  \frac{1}{Ny}\sum_{\begin{bsmallmatrix} m_1 & n_1 \\ m_2 & n_2\end{bsmallmatrix} \in \widehat{L}_{p,q}} \omega(h_\tau,t,g_z)\vphat_{(d_1,d_2)}\left (\begin{bsmallmatrix} m_1 & n_1 \\ m_2 & n_2\end{bsmallmatrix}\right )e\left (-pn_2 \right ),
\end{align}
where $\widehat{L}_{p,q}$ is the coset
\begin{align}
\widehat{L}_{p,q} \coloneqq \left \{ \left . \begin{bmatrix} m_1 & n_1 \\ m_2 & n_2\end{bmatrix} \in V \ \right \vert \  m_1,n_1 \in \Z, \ m_2 \in q + N\Z, \  n_2 \in \frac{1}{N} \Z \right \}.
\end{align}

Let $H_d$ be the Hermite polynomial as in Section \ref{Section : explicit form}.
\begin{lem} \label{Fourier hermite} For any $d \geq 1$, the Fourier transform of $H_d(\sqrt{\pi}(m+n))e^{-\pi m^2-\pi n^2}$ in $n$ is \begin{align} 
\int_\R H_d(\sqrt{\pi}(m+a))e^{-\pi m^2-\pi a^2}e(an)da = 2^di^d \sqrt{\pi}^d \overline{(im+n)}^de^{-\pi m^2-\pi n^2}.
\end{align}
\end{lem}
\begin{proof} Another way to define the Hermite polynomials is by $\left ( -\frac{1}{\sqrt{\pi}} \frac{\partial}{\partial t} \right )^de^{-\pi t^2} =  H_d(\sqrt{\pi}t)e^{-\pi t^2}$. From the addition formula 
\begin{align}
H_d(\sqrt{\pi}(x+y))=\sum_{k=0}^d {d \choose k} H_k(\sqrt{\pi}y)(2\sqrt{\pi}x)^{d-k},
\end{align}
we deduce 
\begin{align}
H_d(\sqrt{\pi}(x+y))e^{-\pi(x^2+y^2)} & =\sum_{k=0}^d {d \choose k} H_k(\sqrt{\pi}y)(2\sqrt{\pi}x)^{d-k}e^{-\pi(x^2+y^2)} \nonumber \\
& = \left (-\frac{1}{\sqrt{\pi}} \frac{\partial}{\partial y} +2\sqrt{\pi}x \right )^de^{-\pi(x^2+y^2)}.
\end{align}
Note that the Fourier transform sends the derivative $\partial/\partial y$ to multiplication by $-2i\pi y'$ and it follows that the partial Fourier transform of $H_d(\sqrt{\pi}(x+y))e^{-\pi(x^2+y^2)}$ is $(2\sqrt{\pi}(x+iy'))^de^{-\pi(x^2+(y')^2)}$.
\end{proof}
From Lemma \ref{Fourier hermite}, it follows that the partial Fourier transform of 
\begin{align}
\varphi_{(d_1,d_2)}\left (\begin{bsmallmatrix} m_1 & n_1 \\ m_2 & n_2\end{bsmallmatrix}\right )=\frac{1}{4\pi}H_{d_1}(\sqrt{\pi}(m_1+n_1))H_{d_2}(\sqrt{\pi}(m_2+n_2))e^{-\pi \lVert m \rVert^2-\pi \lVert n \rVert^2}
\end{align}
is
\begin{align}
\vphat_{(d_1,d_2)}\left ( \begin{bsmallmatrix} m_1 & n_1 \\ m_2 & n_2\end{bsmallmatrix} \right )=- \overline{(n_2+in_1)}^{d_1}\overline{(im_2-m_1)}^{d_2}e^{-\pi \lVert m \rVert^2-\pi \lVert n \rVert^2}.
\end{align}
Note that $d_1+d_2=2$, so that the term $(2i\sqrt{\pi})^{(d_1+d_2)}=-4\pi$ cancels the $1/4\pi$. Applying the formulas of the Weil representation given in \eqref{section weil def}, we find that
\begin{align}
& \omega(h_\tau,t,g_z)\vphat_{(d_1,d_2)}(\begin{bsmallmatrix} m_1 & n_1 \\ m_2 & n_2\end{bsmallmatrix}) \\
&\hspace{3cm} =-v^2t^{d_1-d_2} \frac{\overline{(n_1\tau+n_2)}^{d_1}\overline{(m_2 \tau-m_1)}^{d_2}}{y}e^{-\pi vt^2 \frac{\vert n_1 \tau+n_2 \vert^2}{y}-\pi \frac{v}{t^2} \frac{\vert m_2 \tau -m_1\vert^2}{y}}e^{2i \pi u \langle m,n\rangle},
\end{align}
where $z=u+iv$ and $\tau=x+iy$. Finally, we conclude that $\Theta_{p,q}(z,t,\tau)_{(d_1,d_2)}$ is equal to
\begin{align} \label{theta poisson}
 -\frac{v^2t^{d_1-d_2}}{Ny^2}\sum_{\vbf } \overline{(n_1\tau+n_2)}^{d_1}\overline{(m_2 \tau-m_1)}^{d_2} e\left (u \langle m,n\rangle-pn_2 \right ) e^{-\pi vt^2 \frac{\vert n_1 \tau+n_2 \vert^2}{y}-\pi \frac{v}{t^2} \frac{\vert m_2 \tau -m_1\vert^2}{y}}, \qquad
\end{align}
where the sum is over $\vbf=\begin{bsmallmatrix} m_1 & n_1 \\ m_2 & n_2\end{bsmallmatrix} \in \widehat{L}_{p,q}$.

\section{Integral over modular caps and modular symbols} \label{Section:explicit computations}

The next goal is to compute the image of modular caps and modular symbols under the theta lift $\Ecal$. Using the decomposition of a cycle $\Zcal$ with respect to the isomorphism $H_1(Y;\C) \simeq \Ccal(\C)  \oplus \Mcal \Scal_0(\C)$, together with \eqref{stokes}, we can then express the lift of a cycle $\Zcal$ as a linear combination of lifts of caps and modulars symbols. In order to do that, we need to show that $\Ecal_{1,0}(z,\tau)$ extends to $\overline{\HH}$, in the sense of Definition \ref{definition extension}.

\subsection{Constant term and integral over modular caps} We start by computing the constant term at the cusp $\8$ (in the variable $z$) of $\Ecal_{p,q}(z,\tau)$. Recall that $H_{p,q}^{(2)}(\tau) = G^{(2)}_{q}(\tau)-\delta_{q0}\Ghat^{(2)}_{p}(\tau)$.

\begin{prop} \label{constant term} We have 
\begin{align}
\lim_{v \rightarrow \8}  \Ecal_{p,q}(z,\tau) =H^{(2)}_{p,q}(\tau)du, \qquad \qquad (z=u+iv).
\end{align}
\end{prop}

\begin{proof} We compute the limit as $v \rightarrow \8$ of each of the components
\begin{align}
\Ecal_{p,q}(z,\tau,s)_{(d_1,d_2)} & \coloneqq \int_0^\8 \Theta_{p,q}(z,t,\tau)_{(d_1,d_2)}t^{2s}\frac{dt}{t}.
\end{align}
We start with the $(2,0)$-component. From \eqref{theta poisson} we write
\begin{align} 
\Theta_{p,q}(z,t,\tau)_{(2,0)}&=  -\frac{v^2t^2}{Ny^2}\sum_{\vbf } \overline{(n_1\tau+n_2)}^{2} e\left (u \langle m,n\rangle-pn_2 \right ) e^{-\pi vt^2 \frac{\vert n_1 \tau+n_2 \vert^2}{y}-\pi \frac{v}{t^2} \frac{\vert m_2 \tau -m_1\vert^2}{y}},
\end{align}
where the sum is over $\vbf=[m,n]=\begin{bsmallmatrix} m_1 & n_1 \\ m_2 & n_2\end{bsmallmatrix} \in \widehat{L}_{p,q}$.

For two real numbers $a,b>0$, we define the Bessel function
\begin{align}
K_\nu(a,b) \coloneqq \int_0^\8e^{-(a^2t+b^2/t)}t^\nu\frac{dt}{t},
\end{align} 
which converges for any complex number $\nu \in \C$. We split the sum over regular vectors $[m,n] \in \widehat{L}_{p,q}$, for which $m,n \neq 0$, and singular vectors, for which $m=0$ or $n=0$. 
After integrating the theta series from $0$ to $\8$ and swapping the order of summation and integration (only valid for $\re(s) \gg 0$), we find that
\begin{align}
\Ecal_{p,q}(z,\tau,s)_{(2,0)}=\delta_{q0}A(\tau,v,s)+\sum_{d \in \Z} B^{(d)}(\tau,v,s)e(u d),
\end{align}
where for $d \in \Z$ the regular term is
\begin{align}
B^{(d)}(\tau,v,s)= -\frac{v^2}{2Ny^2}\sum_{\substack{\begin{bsmallmatrix} m_1 & n_1 \\ m_2 & n_2\end{bsmallmatrix} \in \widehat{L}_{p,q} \\ \langle m,n \rangle=d \\ m,n \neq 0}} \overline{(n_1\tau+n_2)}^{2}e(-pn_2) K_{1+s}\left (\pi v \frac{\vert n_1\tau+n_2\vert}{y},\pi v \frac{\vert m_2\tau-m_1 \vert}{y} \right ).
\end{align}
The regular term is rapidly decreasing for any $s \in \C$, so that
\begin{align}
\lim_{v \rightarrow \8} \Ecal_{p,q}(z,\tau,s)_{(2,0)}=\delta_{q0}A(\tau,v,s).
\end{align}

Since $n_1\tau+n_2=0$ if $(n_1,n_2)=(0,0)$, the only singular vectors that contribute to $A(\tau,v,s)$ are the vectors $\begin{bsmallmatrix} m_1 & n_1 \\ m_2 & n_2\end{bsmallmatrix} \in \widehat{L}_{p,q}$ with $(m_1,m_2)=(0,0)$. This is only possible if $q=0$ (otherwise $A=0$), in which case we have
\begin{align}
A(\tau,v,s) & \coloneqq -\frac{v^{1-s}\Gamma \left (1+s \right )}{2N\pi^{1+s}y^{1-s}} \sum_{\substack{n_1 \in \Z \\ n_2 \in \frac{1}{N}{\Z}}}\frac{\overline{n_1\tau+n_2}}{\left (n_1\tau+n_2 \right ) \left \vert n_1\tau+n_2 \right \vert^{2s}} e(-pn_2)\\
& =-\frac{v^{1-s}\Gamma \left (1+s \right )}{2N^{1-2s}\pi^{1+s}y^{1-s}} \sum_{n_1,n_2 \in \Z}\frac{\overline{n_1N\tau+n_2}}{\left (n_1N\tau+n_2 \right ) \left \vert n_1N\tau+n_2 \right \vert^{2s}} e\left (-\frac{pn_2}{N} \right ).
\end{align}
At $s=0$, it follows from the functional equation \eqref{functional equation Gkl} that
\begin{align}
\lim_{s \rightarrow 0} A(\tau,v,s)=2v E^{(1,1)}_{p,0}(N\tau)=2v\Ehat^{(2)}_{p,0}(N\tau)=2v\Ghat^{(2)}_p(\tau)
\end{align}
and we conclude that
\begin{align} \label{constant term 02}
\lim_{v \rightarrow \8} \Ecal_{p,q}(z,\tau,s)_{(2,0)}=\delta_{q0}\Ghat^{(2)}_p du .
\end{align}

The computation for the $(0,2)$-component is analogous and we find that
\begin{align}
\Ecal_{p,q}(z,\tau,s)_{(0,2)}=\widetilde{A}(\tau,v,s)+\sum_{d \in \Z} \widetilde{B}^{(d)}(\tau,v,s) e(u d),
\end{align}
where the sum over the regular terms is rapidly decreasing, and
\begin{align}
\widetilde{A}(\tau,v,s) & \coloneqq -\frac{v^{1+s}}{2N\pi^{1-s}y^{1+s}}\Gamma \left (1-s \right ) \sum_{\substack{m_2 \in q+N\Z \\ m_1 \in \Z}}\frac{\overline{m_2\tau-m_1}}{\left (m_2\tau-m_1 \right ) \left \vert m_2\tau-m_1 \right \vert^{-2s}}.
\end{align}
At $s=0$ it is equal to $\lim_{s \rightarrow 0} \widetilde{A}(\tau,v,s)=2vG^{(2)}_{q}(\tau)$ and it follows that
\begin{align}
\lim_{v \rightarrow \8}\Ecal_{p,q}(z,\tau,s)_{(0,2)}=G^{(2)}_{q}(\tau).
\end{align}

Finally, for the component $(1,1)$ there are no singular contributions, since the term $(n_1\tau+n_2)(m_2 \tau-m_1)$ forces the vanishing of the series whenever $m=0$ or $n=0$. Thus, the sum is rapidly decreasing as $v \rightarrow \8$ and 
\begin{align}
\lim_{v \rightarrow \8} \Ecal_{p,q}(z,\tau,s)_{(1,1)}=0.
\end{align}
\end{proof}

\begin{prop}[Integral over modular caps] \label{constant term 2} Let $\Ccal_r$ be the closed modular cap at the cusp $r=[m:n] \in \PP^1(\Q)$. Write $r=\gamma_r \8$ for some matrix $\gamma_r =\begin{psmallmatrix} m & i \\ n & j\end{psmallmatrix} \in \SL_2(\Z)$. Then
\begin{align}
\int_{\Ccal_r}\Ecal_{1,0}(z,\tau)=H_{j,-n}^{(2)}(\tau).
\end{align}
\end{prop}

\begin{proof}
Since $\int_{\Ccal_\8}du=1$ and the integral of $\Ecal_{p,q}(z,\tau)$ along the modular cap $\Ccal_r$ at the cusp $r$ is given by the constant term (in $z$) of $\Ecal_{p,q}(z,\tau)$ at the cusp $r$, we deduce by Proposition \ref{constant term} that
\begin{align}
\int_{\Ccal_\8} \Ecal_{p,q}(z,\tau)=H^{(2)}_{p,q}(\tau).
\end{align}
From the equivariance \eqref{invariance of Epq} of $\Ecal_{p,q}$, it follows that
\begin{align} \label{invariance of Epq 2}
\int_{\gamma_r \Ccal_\8} \Ecal_{1,0}(z,\tau)=\int_{\Ccal_\8} \gamma_r^\ast \Ecal_{1,0}(z,\tau)=\int_{\Ccal_\8} \Ecal_{j,-n}(z,\tau)=H_{j,-n}^{(2)}(\tau),
\end{align}
since $\gamma_r^{-1} \begin{psmallmatrix}
1 \\
0
\end{psmallmatrix}= \begin{psmallmatrix}
j \\
-n
\end{psmallmatrix}$.
\end{proof}

\begin{rmk} In the sense of Definition \ref{definition extension}, the form $\omega = \Ecal_{0,1}(z,\tau) \in \Omega^1(\HH) \otimes C^\8(\HH)$ extends to the boundary, where the forms at the boundary are 
\begin{align}
\omega^{(r)}=H^{(2)}_{j,-n}(\tau)du \in \Omega^1(B_r) \otimes M_2(\Gamma_1(N)),
\end{align}
with $j,n$ are as in Proposition \ref{constant term 2}. Note that while $\Ecal_{0,1}(z,\tau)$ is only holomorphic (in $\tau$) as a cohomology class, the boundary forms are holomorphic already at the level of differential forms.
\end{rmk}

\subsection{Integral over unimodular symbols}  \label{section unimodular}

It will be enough to compute the integral of $\Ecal_{p,q}$ along the modular symbol $\{0,\8\}$, since the computation over a unimodular symbol $\gamma \{0,\8\}$ follows from the equivariance \eqref{invariance of Epq} of $\Ecal_{p,q}$, as in \eqref{invariance of Epq 2}. Moreover, the restriction of $du$ to $\{0,\8\}$ is zero, so that
\begin{align}
\int_{\{0,\8\}} \Ecal_{p,q}(z,\tau)=\int_0^\8 \Ecal_{p,q}(iv,\tau)_{(1,1)}\frac{dv}{v}
\end{align}
by \eqref{split ephi}. From the Poisson summation \eqref{theta poisson} it follows that the restriction of the theta series along the modular symbol $\{0,\8\}$ is
\begin{align}
& \Theta_{p,q}(iv,t,\tau)_{(1,1)} \\
&=  -\frac{v^2}{Ny^2}\sum_{\begin{bsmallmatrix} m_1 & n_1 \\ m_2 & n_2\end{bsmallmatrix} \in \widehat{L}_{p,q}} \overline{(n_1\tau+n_2)}\overline{(m_2 \tau-m_1)} e\left (-pn_2 \right ) e^{-\pi vt^2 \frac{\vert n_1 \tau+n_2 \vert^2}{y}-\pi \frac{v}{t^2} \frac{\vert m_2 \tau -m_1\vert^2}{y}}.
\end{align}
Thus, we get
\begin{align}
\int_{\{0,\8\}} \Ecal_{p,q}(iv,\tau,s)_{(1,1)}\frac{dv}{v} & = \int_0^\8\int_0^\8 \Theta_{p,q}(iv,t,\tau)_{(1,1)}t^{2s}\frac{dv}{v} \frac{dt}{t} \\
&=-\frac{1}{Ny^2}\sum_{\begin{bsmallmatrix} m_1 & n_1 \\ m_2 & n_2\end{bsmallmatrix} \in \widehat{L}_{p,q}} \overline{(n_1\tau+n_2)}\overline{(m_2 \tau-m_1)} e\left (-pn_2 \right ) \\
& \qquad \qquad \times \int_0^\8\int_0^\8e^{-\pi vt^2 \frac{\vert n_1 \tau+n_2 \vert^2}{y}-\pi \frac{v}{t^2} \frac{\vert m_2 \tau -m_1\vert^2}{y}}v^2t^{2s}\frac{dv}{v} \frac{dt}{t}.
\end{align}
By setting $a=vt^2$ and $b=\frac{v}{t^2}$, we can rewrite the previous integral as
\begin{align} &-\frac{1}{4 Ny^2} \sum_{\begin{bsmallmatrix} m_1 & n_1 \\ m_2 & n_2\end{bsmallmatrix} \in \widehat{L}_{p,q}} \overline{(n_1\tau+n_2)}\overline{(m_2 \tau-m_1)} e\left (-pn_2 \right ) \\
& \qquad \qquad \qquad \qquad \qquad \qquad \times \int_0^\8\int_0^\8e^{-\pi a \frac{\vert n_1 \tau+n_2 \vert^2}{y}-\pi b \frac{\vert m_2 \tau -m_1\vert^2}{y}}a^{1+\frac{s}{2}}b^{1-\frac{s}{2}}\frac{da}{a} \frac{db}{b}\\
&=-\frac{\Gamma \left (1+\frac{s}{2} \right ) \Gamma \left (1-\frac{s}{2} \right )}{4 N\pi^2} \sum_{\begin{bsmallmatrix} m_1 & n_1 \\ m_2 & n_2\end{bsmallmatrix} \in \begin{bsmallmatrix} \Z & \Z \\ q+N\Z & \frac{1}{N}\Z\end{bsmallmatrix}}  \frac {1}{(m_2 \tau-m_1)\vert m_2 \tau-m_1 \vert^{-\frac{s}{2}}} \frac{e\left (-pn_2 \right )}{(n_1\tau+n_2)\vert n_1\tau+n_2 \vert^\frac{s}{2}}.
\end{align}
After comparing with the definition \eqref{defeiseE} of the Eisenstein series, we find that
\begin{align}
\int_{\{0,\8\}} \Ecal_{p,q}(z,\tau)= \Ehat^{(1)}_{q,0}(N\tau)E^{(1)}_{p,0}(N\tau)=E^{(1)}_{q,0}(N\tau)E^{(1)}_{p,0}(N\tau)=G^{(1)}_q(\tau)G^{(1)}_p(\tau),
\end{align}
where we used the functional equation $\Ehat^{(1)}_{r,0}(\tau)=E^{(1)}_{r,0}(\tau)$ from \eqref{functional equation Gkl}

\begin{prop}[Integral over unimodular symbols]  \label{unimodular symbol} Let $\gamma=\begin{psmallmatrix}
a & b \\ c& d
\end{psmallmatrix} \in \SL_2(\Z)$. The integral along the unimodular symbol $\Mcal=\gamma \{0,\8\}$ is
\begin{align}
\int_{\Mcal}\Ecal_{1,0}(z,\tau)=-G_d^{(1)}(\tau)G_{c}^{(1)}(\tau).
\end{align}
\end{prop}
\begin{proof}
From the equivariance \eqref{invariance of Epq} of $\Ecal_{p,q}$, it follows that
\begin{align}
\int_{\gamma \{0,\8\}} \Ecal_{1,0}(z,\tau)=\int_{\{0,\8\}} \gamma^\ast \Ecal_{1,0}(z,\tau)=\int_{\{0,\8\}} \Ecal_{d,-c}(z,\tau),
\end{align}
where
\begin{align}
\gamma^{-1} \begin{pmatrix}
1 \\
0
\end{pmatrix}= \begin{pmatrix}
d \\
-c
\end{pmatrix}.
\end{align} The result follows from the previous computation.
\end{proof}

\subsection{Image of the theta lift}  We can now determine the image of the theta lift
\begin{align}
\Ecal \colon H_1(Y_1(N);\C) \longrightarrow M_2(\Gamma_1(N)).
\end{align} If $\gamma$ is hyperbolic, let $[b_0, \dots,b_n]$ be the continued fraction expansion of $\frac{a}{c}=\gamma \8$ as in \eqref{continued fraction}. Let $\frac{p_0}{q_0},\dots,\frac{p_n}{q_n}$ be the convergents, and $m=p_{n-1}d-bq_{n-1}$. We set $(p_{-1},q_{-1})\coloneqq (1,0)$.
If $\gamma$ is parabolic, let $b(\gamma) \in \Z$ be as in \eqref{def bgamma}.

\begin{thm} \label{main them text} If $\gamma$ is a parabolic matrix in $\Gamma_1(N)$ stabilizing the cusp $r=[m:n]$, then 
\begin{align} \label{parabolic cycles}
\Ecal(\Zcal_\gamma)=b(\gamma)  H_{j,-n}^{(2)},
\end{align}
where $i,j$ are integers such that $mj-in=1$. If $\gamma = \begin{psmallmatrix} a & b \\ c & d \end{psmallmatrix}$ is a hyperbolic matrix in $ \Gamma_1(N)$, then
\begin{align}
\Ecal(\Zcal_\gamma) =(b_0+bq_{n-1}-p_{n-1}d)H_{1,0}^{(2)} + \sum_{k=0}^{n-1} b_{k+1}H^{(2)}_{q_{k-1},q_k} - \sum_{k=0}^n G_{q_k}^{(1)}G_{q_{k-1}}^{(1)}.
\end{align}

\end{thm}

\begin{proof}Recall from Theorem \ref{split cycle} that we can write the cycle as
\begin{align}
\Zcal_\gamma  = \begin{cases} \ 
  b(\gamma) \Ccal_r & \ \ \textrm{if} \ \gamma \ \textrm{is parabolic}, \\[1em]
 \ (b_0+bq_{n-1}-p_{n-1}d)\Ccal_\8 + \sum_{k=0}^{n-1} b_{k+1} \Ccal_{\gamma_k\8} +  \sum_{k=0}^n \gamma_k \{ 0,\8\} & \ \ \textrm{if} \ \gamma \ \textrm{is hyperbolic}
\end{cases}
\end{align}
under the isomorphism $H_1(Y;\C) \simeq \Ccal(\C)  \oplus \Mcal \Scal_0(\C)$.  The integral along the unimodular symbol $\gamma_k \{0,\8\}$ is
\begin{align}
\int_{\gamma_k \{0,\8\}} \Ecal_{1,0}(z,\tau)=-G_{q_k}^{(1)}(\tau)G_{q_{k-1}}^{(1)}(\tau)
\end{align}
by Proposition \ref{unimodular symbol}, where $\gamma_k=
\begin{psmallmatrix}
-p_k & p_{k-1} \\ -q_k & q_{k-1}
\end{psmallmatrix}$.  
If $\gamma_r \in \SL_2(\Z)$ is a matrix such that $[m:n]=r=\gamma_r \8$, then $\gamma_r =\begin{pmatrix} m & i \\ n & j\end{pmatrix}$ for some integers $i,j$ such that $mj-in=1$. By Proposition \ref{constant term 2}, we have 
\begin{align}
\int_{\Ccal_r} \Ecal_{1,0}(z,\tau)=H_{j,-n}^{(2)}(\tau).
\end{align} 
In particular, the integral along the modular cap $\Ccal_{\gamma_k\8}$ is
\begin{align}
\int_{\Ccal_{\gamma_k\8}}\Ecal_{1,0}(z,\tau)=H^{(2)}_{q_{k-1},q_k}(\tau).
\end{align}
For $\gamma_r=\id$, the integral along $\Ccal_\8$ is
\begin{align}
\int_{\Ccal_\8} \Ecal_{1,0}(z,\tau)=H_{1,0}^{(2)}(\tau).
\end{align}
\end{proof}

\begin{mydef} Let $\Hcal^{(2)} \subseteq E_2(\Gamma_1(N))$ be the subspace 
\begin{align}
\Hcal^{(2)} \coloneqq \Span_\C\left \{ \left . H_{p,q}^{(2)}(\tau) \ \right  \vert \ (p,q) \not\equiv (0,0) \mod{N} \right \}.
\end{align}
Let $\Hcal^{(1,1)} \subseteq M_2(\Gamma_1(N))$ be the subspace
\begin{align}
\Hcal^{(1,1)} \coloneqq \Span_\C\left \{ \left . G^{(1)}_a(\tau)G^{(1)}_b(\tau) \ \right \vert \ a,b \in \Z/N\Z \right \}.
\end{align}
\end{mydef}


By \eqref{parabolic cycles}, the integral of $\Ecal_{1,0}$ over the modular caps are Eisenstein series. Hence, the restriction of $\Ecal$ to $\Ccal(\C)$ is a linear map
\begin{align}
\restr{\Ecal}{\Ccal(\C)} \colon \Ccal(\C) \longrightarrow E_2(\Gamma_1(N)).
\end{align}
\begin{prop} \label{im eis proposition} We have $\Ecal \left (\Ccal(\C) \right )=\Hcal^{(2)}$.
\end{prop}
\begin{proof}
Note that $\Ecal \left (\Ccal(\C) \right ) = \Span_\C\left \{ \left . \int_{\Ccal_r} \Ecal_{1,0} \ \right \vert \ r \in C_N\right \}$. By Proposition \ref{constant term 2}, we have $\int_{\Ccal_r}\Ecal_{1,0}(z,\tau)=H_{j,-n}^{(2)}(\tau)$ with $(j,n)$ a pair of coprime integers. In particular, we have $(j,n) \not \equiv (0,0) \mod N$ and thus $\Ecal \left (\Ccal(\C) \right ) \subseteq \Hcal^{(2)}$. On the other hand, suppose that $H^{(2)}_{p,q} \in \Hcal^{(2)}$. If $q \equiv 0 \mod N$, then $N$ is coprime to $p$ and we can find an integer $l$ such that $l \equiv (r-1)N^{-1} \mod p$. This means that 
$1=(q-Nl)+kp$ for some integer $k$, and thus $(j,-n) \coloneqq (p,q-Nl)$ is a pair of coprime integers that is congruent to $(p,q)$ modulo $N$, so that $H^{(2)}_{p,q}(\tau)=H^{(2)}_{j,-n}(\tau)$. Since $j$ and $n$ are coprime, we can find a matrix $\gamma=\begin{psmallmatrix}\ast & \ast \\ n & j \end{psmallmatrix} \in \SL_2(\Z)$. By Proposition \ref{constant term 2} we can then find a modular cap such that its image is $H^{(2)}_{p,q}$, and thus we have $H^{(2)}_{p,q} \in \Ecal\left (\Ccal(\C) \right )$. If $q \not \equiv 0 \mod N$, then $H^{(2)}_{p,q}(\tau)=G^{(2)}_q(\tau)$ does not depend on $p$. Hence, we have $H^{(2)}_{p,q}(\tau)=H^{(2)}_{1,q}(\tau)$, and since $1$ and $q$ are coprime, it follows from the previous argument that $H^{(2)}_{1,q} \in \Ecal \left (\Ccal(\C) \right )$. This shows that $\Hcal^{(2)} \subseteq \Ecal \left (\Ccal(\C) \right )$.
\end{proof}
\begin{cor} \label{image Eisenstein} We have
\begin{align}
\Hcal^{(2)} \oplus S_{2,\rk=0}^\new(\Gamma_1(N))  \subseteq \im(\Ecal) \subseteq \Span_\C\{ \Hcal^{(2)}, \Hcal^{(1,1)} \},
\end{align}
where the first inclusion is an equality when $N$ is a prime.
\end{cor}

\begin{proof} First, note that the forms $H^{(2)}_{1,0}$ and $H^{(2)}_{q_{k-1},q_k}$ that appear in Theorem  \ref{main them text} are all in $\Hcal^{(2)}$, since $q_{k-1},q_k$ are coprime integers. This proves the second inclusion $\im(\Ecal) \subseteq \Span_\C\{ \Hcal^{(2)}, \Hcal^{(1,1)} \}$. 

For the first inclusion, recall the decomposition
\begin{align}
H_1(Y_1(N);\C)=\Ccal(\C) \oplus \Span_\C\{ \Zcal_f, \overline{\Zcal_f} \ \vert \ f \in \Bcal \}.
\end{align} from \eqref{splitting cusp eis homology}. By Theorem \ref{surjective morphism}, the image of $\Span_\C\{ \Zcal_f, \overline{\Zcal_f} \ \vert \ f \in \Bcal \}$ contains $S_{2,\rk=0}^\new(\Gamma_1(N))$, with inclusion when $N$ is prime. On the other hand, the image of $\Ccal(\C)$ is $\Hcal^{(2)}$, by Theorem \ref{im eis proposition}.
\end{proof}

\begin{rmk} It follows from the proof of Theorem \ref{im eis proposition} that $\Hcal^{(2)}$ is spanned by by the weight $2$ Eisenstein series $H^{(2)}_{1,q}(\tau)=G^{(2)}_q(\tau)$ with $q \not \equiv 0 \mod{N}$, and $H^{(2)}_{p,0}(\tau)=G^{(2)}_{0}(\tau)-\Ghat^{(2)}_{p}(\tau)$ with $p \in (\Z/N\Z)^\times$.
\end{rmk}

\subsection{Diagonal restrictions of Eisenstein series} Recall that a matrix $\gamma$ is hyperbolic if $\vert \tr(\gamma) \vert>2$, and parabolic if $\tr(\gamma)=2$. 

\begin{prop} \label{prop im hyp} We have $\im(\Ecal)= \Span_\C \{ \Ecal(\Zcal_\gamma) \ \vert \ \gamma \in \Gamma_1(N) \ \textrm{hyperbolic}\}$.
\end{prop}
\begin{proof} It follows from the fact that any parabolic matrix can be written as a product of hyperbolic matrices. More precisely, if $\gamma=\begin{psmallmatrix}
a & b \\ c & t-a
\end{psmallmatrix}$ is a parabolic matrix with $\vert t \vert=\vert \tr(\gamma) \vert =2$, then we can write
\begin{align}
\gamma=\gamma \gamma_1 \gamma_1^{-1}, \qquad \gamma_1=\begin{pmatrix}
1+m^2 & m \\ m & 1
\end{pmatrix} \in \Gamma_1(N)
\end{align}
where $m$ is any integer divisible by $N$. We have 
\begin{align}
\tr(\gamma \gamma_1)=t+(b+c)m+am^2.
\end{align}
For $m$ large enough, both $\gamma_1$ and $\gamma \gamma_1$ are hyperbolic, and in homology we have
\begin{align}
\Zcal_{\gamma}=\Zcal_{\gamma \gamma_1}-\Zcal_{\gamma_1}.
\end{align}
Thus, $\Ecal(\Zcal_{\gamma})=\Ecal(\Zcal_{\gamma \gamma_1})-\Ecal(\Zcal_{\gamma_1})$ is a linear combination of images of hyperbolic matrices.
\end{proof}

Let $\gamma=\begin{psmallmatrix}
a & b \\ c & d
\end{psmallmatrix}$ be a hyperbolic matrix in $\Gamma_1(N)$, and $F \coloneqq \Q(\gamma)=\{a+b\gamma \ \vert \ a,b \in \Q \}$ the real quadratic field generated by $\gamma$. Let $D \coloneqq \tr(\gamma)^2-4$. We have
\begin{align}
\gamma & = \delta_0^{-1}\begin{pmatrix}
\epsilon & 0 \\ 0 & \epsilon'
\end{pmatrix}\delta_0,
\end{align}
where
\begin{align}
\delta_0 \coloneqq \begin{pmatrix} 1 & \nu \\ 1 & \nu' \end{pmatrix} \in \GL_2(\R)^+, \quad \textrm{with} \ \  \nu = \frac{\epsilon-d}{c}, \quad \ \ \epsilon=\frac{a+d+\sqrt{D}}{2},
\end{align}
and $\epsilon'=(a+d-\sqrt{D})/2$ its Galois conjugate. We have an isomorphism 
\begin{align} \label{field iso}
\Q(\gamma)\longrightarrow \Q(\epsilon)=\Q(\sqrt{D}) , \quad \delta_0^{-1}\begin{pmatrix}
\mu & 0 \\ 0 & \mu'
\end{pmatrix}\delta_0 \longmapsto \mu.
\end{align}
The linear map $\delta_0 \colon \R^2 \longrightarrow \R^2$ is $F$-equivariant, where $F$ acts on the left as a matrix in $F = \Q(\gamma) \subset \GL_2(\Q)$, and on the right $F \simeq \Q(\epsilon)$ acts by $\mu \cdot (x,y)=(\mu x, \mu' y)$. The lattice $\Z^2$ is sent to $\delta_0\Z^2=\sigma(\ffrak)=\left \{ (\mu, \mu') \in \R^2 \ \vert \ \mu \in \ffrak \right \}$, where $\ffrak \coloneqq \Z+\nu \Z$. Let $\Ocal \subseteq \Ocal_F$ be the largest order preserving $\ffrak$. It corresponds to the stabilizer in $\Q(\gamma)$ of the lattice $\Z^2$, so that 
\begin{align} \label{some iso}
\Q(\gamma) \cap \SL_2(\Z) \simeq \Ocal^1,
\end{align} 
where $\Ocal^1 \subset \Ocal^\times$ is the subgroup of elements of norm $1$. On the other hand, $\Q(\gamma) \cap \Gamma_1(N)$ is the stabilizer of the coset $\begin{psmallmatrix}
1+N\Z \\ N\Z
\end{psmallmatrix} \cong 1+N\ffrak$, so that the isomorphism \eqref{some iso} restricts to
\begin{align}
\Q(\gamma) \cap \Gamma_1(N) \cong U,
\end{align}where $U \coloneqq \Ocal^1 \cap (1+N\ffrak) \subset \Ocal^1$ is the stabilizer of the coset $1+N\ffrak$.  Finally, let $U^+ \coloneqq U \cap \Ocal^{\times,+}$, where $\Ocal^{\times,+} \subset \Ocal^1$ is the subgroup of totally positive units in $\Ocal^\times$.
\begin{lem} \label{lemma primitive} If $\gamma$ is a primitive hyperbolic matrix in $\Gamma_1(N)$ with positive eigenvalues, then $U^+=\epsilon^\Z \cong \gamma^\Z$.
\end{lem}
\begin{proof}
By Dirichlet's unit theorem, the group of units $\Ocal^{\times,+}$ is cyclic and generated by a fundamental unit $\epsilon_0$, which is an eigenvalue of a primitive hyperbolic matrix
\begin{align}
\gamma_0 \coloneqq \delta_0^{-1}\begin{pmatrix}
\epsilon_0 & 0 \\ 0 & \epsilon'_0
\end{pmatrix}\delta_0 \in \SL_2(\Z).
\end{align}
As $U^+$ is a subgroup of $\Ocal^{\times,+} \simeq \gamma_0^\Z$, it is generated by some power $\gamma_1 \coloneqq \gamma_0^l \in \Gamma_1(N)$, for some $l \geq 1$, and with eigenvalue $\epsilon_1$. The matrix $\gamma \in \Q(\gamma) \cap \Gamma_1(N) \simeq U$ lies in $U^+\cong \gamma_1^\Z $, since we assumed that it has positive eigenvalues. Thus, we have $\gamma=\gamma_1^m$ for some $m \geq 1$. Since $\gamma$ is primitive in $\Gamma_1(N)$, we have $m=1$ and $\gamma=\gamma_1$.
\end{proof}

Let us now compute the periods 
\begin{align}
 \Ecal(\Zcal_\gamma)=\int_{\Zcal_\gamma} \Ecal_{1,0}=\int_{z_0}^{\gamma z_0}\Ecal_{1,0}
\end{align} along the $1$-cycles $\Zcal_\gamma=\{z_0,\gamma z_0 \}$. The computation is similar to \cite[Section.~4.11.2]{rbrsln}.

If $\gamma$ is not primitive, then $\gamma=\gamma_1^m$ for some primitive hyperbolic matrix $\gamma_1$ and $\Ecal(\Zcal_\gamma)=m\Ecal(\Zcal_{\gamma_1})$. Without loss of generality, let us assume that $\gamma$ is primitive. As $-\id_2 \in \Gamma_1(N)$ acts trivially on $\HH$, we have $\Zcal_{-\gamma}=\Zcal_\gamma$. By replacing $\gamma$ by $-\gamma$ if necessary, we can furthermore assume that $\gamma$ has positive eigenvalues. Since the integral along $\{z_0,\gamma z_0 \}$ does not depend on the basepoint, we pick $z_0=\delta^{-1}_0i$ so that
\begin{align}
\{z_0,\gamma z_0 \}= \delta^{-1}_0\{i,\epsilon^2i \}.
\end{align}Moreover, the segment $\{i,\epsilon^2i \}$ is contained in $\{0,\8\}$, so that only the $dv$-component contributes, as in Section \ref{section unimodular}.  We have
\begin{align}
\int_{z_0}^{\gamma z_0}\Ecal_{1,0}(z,\tau)=\int_1^{\epsilon^2}(\delta_0^{-1})^\ast \Ecal_{1,0}(iv,\tau)_{(1,1)}\frac{dv}{v}.
\end{align}
Since the form $\varphi$ satisfies the equivariance \eqref{equivariance}, we get
\begin{align} \label{an integral diag restr}
\int_{z_0}^{\gamma z_0}\Ecal_{1,0}(z,\tau,s)= \int_0^\8 \int_1^{\epsilon^2}\frac{1}{y} \sum_{(m,n) \in \rho_{\delta_0}L_{1,0}} \omega(g_{iv},t,h_\tau)\varphi_{(1,1)}(\vbf)t^{2s} \frac{dt}{t}\frac{dv}{v}.
\end{align}
The action of $\rho_g$ on $\begin{bsmallmatrix} m_1 & n_1 \\ m_2 & n_2\end{bsmallmatrix}$ was given by $[gm,g^{-T}n]$. This action becomes $[gm,gn]$ after performing the partial Fourier transform in $n$

\begin{align}
\widetilde{\varphi}\left (\begin{bmatrix}
m_1 & n_1 \\ m_2 & n_2
\end{bmatrix} \right ) \coloneqq \int_{\R^2} \varphi \left (\begin{bmatrix}
m_1 & a \\ m_2 & b
\end{bmatrix} \right )e(an_1+bm_2) dadb.
\end{align}
Note that whereas the Fourier transform $\vphat$ in \eqref{fourier transform 1} was taken in the variables $m_1,n_2$, here it is taken in $n_1,n_2$. By Lemma \ref{Fourier hermite}, one can similarly compute
\begin{align}
\widetilde{\varphi}_{(1,1)}(\begin{bsmallmatrix} m_1 & n_1 \\ m_2 & n_2\end{bsmallmatrix})=- \overline{(m_1i+n_1)}^{d_1}\overline{(m_2i+n_2)}^{d_2}e^{-\pi \lVert m \rVert^2-\pi \lVert n \rVert^2}.
\end{align}
Applying Poisson summation, we find that (along the geodesic $\{ 0,\8\}$)
\begin{align}
& \frac{1}{y} \sum_{\vbf \in \rho_{\delta_0}L_{1,0}} \omega(g_{iv},t,h_\tau)\varphi_{(1,1)}(\vbf) \\
& \qquad \qquad = -D^{\frac{1}{2}} \sum_{\vbf \in \delta_0L_{1,0}} \frac{\overline{(m_1\tau+n_1)}\overline{(m_2\tau+n_2)}}{y^2t^4} e^{-\pi \frac{1}{vt^2} \frac{\vert m_1 \tau+n_1 \vert^2}{y}-\pi  \frac{v}{t^2}\frac{\vert m_2\tau+n_2\vert^2}{y}},
\end{align}
where $\vert \det(\delta_0)\vert=D^{\frac{1}{2}}$, $z=u+iv$ and $\tau=x+iy$. The sum is now over the coset
\begin{align}
\delta_0\begin{psmallmatrix} 1+N\Z & \Z \\
N \Z & \Z \end{psmallmatrix} \simeq (1+N\ffrak) \times \ffrak
\end{align}
and the integral \eqref{an integral diag restr} is equal to
\begin{align}
-2D^{\frac{1}{2}}\int_0^\8 \int_1^{\epsilon}\sum_{ (m,n) \in (1+N\ffrak) \times \ffrak} \frac{\overline{(m\tau+n)}\overline{(m'\tau+n')}}{y^2} e^{-\pi \frac{1}{w^2t^2} \frac{\vert m \tau+n \vert^2}{y}-\pi  \frac{w^2}{t^2}\frac{\vert m'\tau+n'\vert^2}{y}} t^{2s-4}\frac{dt}{t}\frac{dw}{w},
\end{align}
after setting $w^2=v$. Since we assumed that $\gamma$ is primitive and has positive eigenvalues, we have $U^+=\epsilon^\Z$ by Lemma \ref{lemma primitive}. By unfolding and changing variables to $(a,b)=(\frac{1}{w^2t^2},\frac{w^2}{t^2})$, we get
\begin{align} \label{equation blabla 1}
& \frac{D^{\frac{1}{2}}}{4}\int_0^\8 \int_0^{\8}\sum_{ (m,n) \in (1+N\ffrak) \times \ffrak/U^+} \frac{\overline{(m\tau+n)}\overline{(m'\tau+n')}}{y^2} e^{-\pi a \frac{\vert m \tau+n \vert^2}{y}-\pi  b\frac{\vert m'\tau+n'\vert^2}{y}} (ab)^{1-\frac{s}{2}}\frac{da}{a}\frac{db}{b} \nonumber \\
& = \frac{D^{\frac{1}{2}} \Gamma\left (1-\frac{s}{2} \right )^2}{4\pi^{2-s}} \sum_{(m,n) \in (1+N\ffrak) \times \ffrak / U^+}  \frac{y^{-s}}{(m\tau+n)(m'\tau+n')\vert m\tau+n\vert^{-s} \vert m'\tau+n'\vert^{-s}},
\end{align}
where $U^+$ acts diagonally by $\epsilon(m,n)=(\epsilon m, \epsilon n)$.
Consider the real-analytic Hilbert-Eisenstein series
\begin{align} \label{sum eisenstein hecke}
& E_{1,\ffrak}^{(k)}(\tau,\tau',s) \\
& \qquad = \frac{(-1)^k D^{k-\frac{1}{2}}}{(2i \pi)^{2k}}\sideset{}{^\prime}\sum_{(m,n) \in ( 1+N\ffrak) \times N\ffrak / U^+} \frac{(yy')^s}{(m\tau+n)^k(m'\tau'+n')^k\vert m\tau+n \vert^{2s} \vert m'\tau'+n'\vert^{2s}},
\end{align}
 of weight $(k,k)$ for $\Gamma_1(N\ffrak) \subset \SL_2(F)$. Similar Eisenstein series have been considered by Hecke \cite[p.~223]{H24}. The sum \eqref{sum eisenstein hecke} converges for $\re(s)>1-\frac{k}{2}$ and admits an analytic continuation to the entire plane.  We let $E_{1,\ffrak}^{(k)} \left (\tau,\tau'\right )\coloneqq E_{1,\ffrak}^{(k)} \left (\tau,\tau',0\right )$ be the value at $s=0$, which is a holomorphic Hilbert modular form of weight $(k,k)$. Note that since $\Z \subset \ffrak$, we have $\Gamma_1(N) \subset \Gamma_1(N\ffrak) \cap \SL_2(\Z)$. Hence, the diagonal restriction (obtained by setting $\tau=\tau'$) is a holomorphic modular form
\begin{align}
E_{1,\ffrak}^{(k)} \left (\tau,\tau\right ) \in M_{2k}(\Gamma_1(N)).
\end{align}
By comparing \eqref{equation blabla 1} and \eqref{sum eisenstein hecke}, we find that
\begin{align}
\int_{z_0}^{\gamma z_0}\Ecal_{1,0}(z,\tau,s) =\pi^s \Gamma\left (1-\frac{s}{2} \right )^2 E_{1,\ffrak}^{(1)} \left (\tau,\tau',-\frac{s}{2} \right ).
\end{align}
\begin{prop} Let $\gamma=\gamma_1^m$ be a hyperbolic matrix, for some primitive hyperbolic matrix $\gamma_1 \in \Gamma_1(N)$. Then
\begin{align}
\Ecal(\Zcal_\gamma) = mE_{1,\ffrak}^{(1)} \left (\tau,\tau\right ).
\end{align}
\end{prop}

\subsection{Linear relations between Eisenstein series}

In this last section, we show how to recover linear relations between the Eisenstein series.

\begin{thm} \label{theorem relation}
Let $n_1,n_2,n_3$ be three integers coprime to $N$ that satisfy $n_1+n_2+n_3 \equiv 0 \mod N$. Then 
\begin{align}
G_{n_1}^{(1)}G_{n_2}^{(1)}+G_{n_2}^{(1)}G_{n_3}^{(1)}+G_{n_3}^{(1)}G_{n_1}^{(1)}=G_{n_1}^{(2)}+G_{n_2}^{(2)}+G_{n_3}^{(2)}.
\end{align}
\end{thm}
\begin{proof} The Eisenstein series only depend on the residue of $n_i$ modulo $N$, so we can replace $n_i$ by $n_i-k_iN$. First, we can pick $k_2$ such that $n_2-k_2N$ and $n_1$ are coprime integers: since $n_1$ and $N$ are coprime, there exists an integer $k_2$ such that $k_2 \equiv (n_2-1)N^{-1} \mod{n_1}$, and thus $n_2-1 \equiv k_2N \mod n_1$. It follows that $\gcd(n_1,n_2-k_2N)=1$. Moreover, by assumption, we have $n_1+n_2+n_3=k_3N$ for some $k_3$. By replacing $n_3$ with $n_3'=n_3-k_3N$, we can assume that the $n_i's$ are such that
\begin{align}
n_1+n_2+n_3=0
\end{align}
and $n_1,n_2$ coprime. Since $n_1,n_2$ are coprime, we can then find $m_1,m_2$ such that 
\begin{align}
\gamma_{12}\coloneqq\begin{pmatrix}
m_2 & m_1 \\ -n_2 & -n_1
\end{pmatrix} \in \SL_2(\Z).
\end{align}
Setting $m_3 \coloneqq -m_1-m_2$, we find that the matrices
\begin{align}
\gamma_{23}\coloneqq\begin{pmatrix}
m_3 & m_2 \\ -n_3 & -n_2
\end{pmatrix}, \quad \gamma_{31}\coloneqq\begin{pmatrix}
m_1 & m_3 \\ -n_1 & -n_3
\end{pmatrix}, 
\end{align}
are also in $\SL_2(\Z)$. If $\Mcal_{ij} \coloneqq \gamma_{ij} \{0,\8\}$ is the corresponding unimodular symbol, then
\begin{align}
\Mcal_{12}=\left \{ r_1, r_2\right \}, \quad \Mcal_{23}=\left \{ r_2, r_3\right \}, \quad \Mcal_{31}=\left \{ r_3, r_1\right \},
\end{align}
with $r_i=-\frac{m_i}{n_i}$. Let $\Tcal$ be the triangle having the oriented modular symbols as sides, and closed by adding the modular caps
\begin{align}
\Ccal_1=[r_3,r_2]_{r_1}, \quad \Ccal_2=[r_1,r_3]_{r_2}, \quad \Ccal_3=[r_2,r_1]_{r_3};
\end{align} 
see Figure \ref{moving to boundary 4} below. 
\begin{figure}[h] 
\centering
\captionsetup{width=.75\textwidth,font={small,it}}
\includegraphics[scale=0.40]{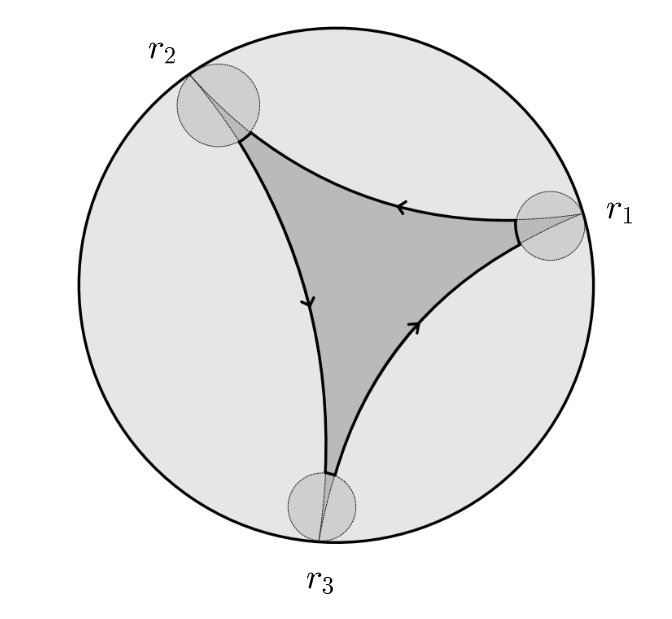}
\caption{A hyperbolic triangle with unimodular sides and closed by modular caps.}
\label{moving to boundary 4}
\end{figure}
We get
\begin{align}
\Mcal_{12}+\Mcal_{23}+\Mcal_{31}+\Ccal_1+\Ccal_2+\Ccal_3=\partial \Tcal = 0 \in H_1(\YBSN;\Z),
\end{align}
from which it follows that
\begin{align}
\int_{\Mcal_{12}+\Mcal_{23}+\Mcal_{31}}\Ecal_{1,0}(z,\tau)=-\int_{\Ccal_1+\Ccal_2+\Ccal_3}\Ecal_{1,0}(z,\tau).
\end{align}
From Proposition \ref{unimodular symbol}, we deduce that
\begin{align}
\int_{\Mcal_{ij}}\Ecal_{1,0}(z,\tau)=-G_{n_i}^{(1)}(\tau)G_{n_j}^{(1)}(\tau).
\end{align}
For the modular caps, notice that
\begin{align}
\gamma_{ij}(0)=-\frac{m_i}{n_i}=r_i, \quad \gamma_{ij}(\8)=-\frac{m_j}{n_j}=r_j, \quad \gamma_{ij}(1)=-\frac{m_k}{n_k}=r_k,
\end{align}
where $k$ is such that $\{i,j,k\}=\{1,2,3\}$. It follows that
\begin{align}
\Ccal_1=[r_3,r_2]_{r_1}=\gamma_{31}[0,1]_\8,
\end{align}
and similarly $\Ccal_2=\gamma_{12}[0,1]_\8$, and $\Ccal_3=\gamma_{23}[0,1]_\8$. From Proposition \ref{constant term 2}, we deduce that
\begin{align}
\int_{\Ccal_1}\Ecal_{1,0}(z,\tau)&=H^{(2)}_{-n_3,n_1}(\tau)=G^{(2)}_{n_1}(\tau)-\delta_{n_3 0}\Ghat^{(2)}_{-n_3}(\tau)=G_{n_1}^{(2)}(\tau),
\end{align}
and similarly $\int_{\Ccal_i}\Ecal_{1,0}(z,\tau)=G_{n_i}^{(2)}(\tau)$ for $i=2$ or $3$.
\end{proof}

The condition on $n_1,n_2,n_3$ guarantees the existence of a triangle with unimodular sides. The rest of the argument extends naturally to an arbitrary polygon whose vertices are the cusps $\frac{m_1}{n_1}, \dots,\frac{m_d}{n_d}$ (with $m_i \neq 0 \mod N$), and the sides are unimodular symbols. The integral along each modular symbol is $-G_{n_i}^{(1)}G_{n_{i+1}}^{(1)}$, and over the modular cap at $\frac{m_i}{n_i}$ the integral is $G_{n_i}^{(2)}$. 

\begin{cor} \label{corpolygon} If there exists a polygon joining cusps $\frac{m_1}{n_1}, \dots,\frac{m_d}{n_d}$ by unimodular symbols, then
\begin{align}
G_{n_1}^{(1)}G_{n_2}^{(1)}+G_{n_2}^{(1)}G_{n_3}^{(1)}+\cdots+G_{n_d}^{(1)}G_{n_1}^{(1)}=G_{n_1}^{(2)}+\cdots+G_{n_d}^{(2)}.
\end{align}
\end{cor}

\printbibliography
\end{document}